\documentclass[number,citesort,dvips]{arxbj}
\usepackage{mathbh}
\usepackage{upgreek}

\aid{0}
\volume{17}
\issue{2}
\pubyear{2011}
\firstpage{687}
\lastpage{713}
\doi{10.3150/10-BEJ288}

\makeatletter

\newtheorem{theorem}{Theorem}
\newtheorem{lemma}{Lemma}
\newtheorem{corollary}{Corollary}
\newtheorem{proposition}{Proposition}
\newremark{remark}{Remark}

\makeatother

\begin{document}
\begin{frontmatter}

\title{Margin-adaptive model selection in statistical learning}
\runtitle{Margin-adaptive model selection}

\begin{aug}
\author[1]{\fnms{Sylvain} \snm{Arlot}\thanksref{1}\corref{}\ead[label=e1]{sylvain.arlot@ens.fr}}
\and
\author[2]{\fnms{Peter L.} \snm{Bartlett}\thanksref{2}\ead[label=e2]{bartlett@cs.berkeley.edu}}

\runauthor{S. Arlot and P. L. Bartlett}

\address[1]{CNRS,
Willow Project-Team,
Laboratoire d'Informatique de l'Ecole Normale Superieure
(CNRS/ENS/INRIA UMR 8548),
23, avenue d'Italie,
CS 81321,
75214 Paris Cedex 13,
France.
\printead{e1}}
\address[2]{Computer Science Division and Department of Statistics,
University of California, Berkeley,
367 Evans Hall \#3860,
Berkeley, CA 94720-3860,
USA.
\printead{e2}}
\end{aug}


%
\begin{abstract}
A classical condition for fast learning rates is the margin condition,
first introduced by Mammen and Tsybakov. We tackle in this paper the
problem of adaptivity to this condition in the context of model
selection, in a general learning framework. Actually, we consider a
weaker version of this condition that allows one to take into account
that learning within a small model can be much easier than within a large
one. Requiring this ``strong margin adaptivity'' makes the model
selection problem more challenging. We first prove, in a general
framework, that some penalization procedures (including local
Rademacher complexities) exhibit this adaptivity when the models are
nested. Contrary to previous results, this holds with penalties that
only depend on the data. Our second main result is that strong margin
adaptivity is not always possible when the models are not nested: for
every model selection procedure (even a randomized one), there is a
problem for which it does not demonstrate strong margin adaptivity.
\end{abstract}

%
\begin{keyword}
\kwd{adaptivity}
\kwd{empirical minimization}
\kwd{empirical risk minimization}
\kwd{local Rademacher complexity}
\kwd{margin condition}
\kwd{model selection}
\kwd{oracle inequalities}
\kwd{statistical learning}
\end{keyword}

\end{frontmatter}

\section{Introduction} \label{sec.setting}
We consider in this paper the model selection problem in a general
framework. Since our main motivation comes from the supervised binary
classification setting, we focus on this framework in this
introduction. Section \ref{sec.gal-set} introduces the natural
generalization to empirical (risk) minimization problems, which we
consider in the remainder of the paper.

We observe independent realizations $(X_i,Y_i) \in\mathcal{X}\times
\mathcal{Y}$ for
$i =
1, \ldots, n$ of a random variable with distribution $P$, where
$\mathcal{Y}=
\{ 0,1 \}$. The goal is to build a (data-dependent) predictor $t$
(i.e., a~measurable function $\mathcal{X}\mapsto\mathcal{Y}$) such
that $t(X)$ is as often
as possible equal to $Y$, where $(X,Y) \sim P$ is independent
from the
data. This is the \textit{prediction} problem, in the setting of
supervised binary classification. In other words, the goal is to find
$t$ minimizing the prediction error $P \gamma
( t;\cdot) :=\mathbb{P}_{(X,Y) \sim P} ( t(X) \neq Y )$,
where $\gamma$ is the 0--1 loss.

The minimizer $s$ of the prediction error, when it exists, is
called the Bayes predictor. Define the regression function $\eta(X) =
\mathbb{P}_{(X,Y) \sim P} ( Y = 1 \vert X )$. Then, a classical
argument shows that $s(X) = \mathbh{1}_{\eta(X) \geq1/2}$. However,
$s$ is unknown, since it depends on the unknown distribution $P$.
Our goal is to build from the data some predictor $t$ minimizing the
prediction error, or equivalently the excess loss $\ell(s, t ) :=P
\gamma( t ) - P \gamma( s)$.

A classical
approach to the prediction problem is \textit{empirical risk
minimization}. Let $ P_n = n^{-1} \sum_{i=1}^n \delta_{(X_i,Y_i)} $ be
the empirical measure and $S_m$ be any set of predictors, which is
called a \textit{model}. The \textit{empirical risk minimizer} over
$S_m$ is then defined as
\[
\hat{s}_m \in\arg\min_{t \in S_m} P_n \gamma( t ) = \arg\min
_{t \in S_m} \Biggl\{\frac{1}{n} \sum_{i=1}^n \mathbh{1}_{t(X_i)
\neq
Y_i}\Biggr\}.
\]
We expect that the risk of $\hat{s}_m$ is close to that of
\[
s_m \in\arg\min_{t \in S_m} P\gamma( t ),
\]
assuming that such a minimizer exists.

\subsection{Margin condition} \label{sec.setting.margin}
Depending on some properties of $P$ and the complexity of $S_m$, the
prediction error of $\hat{s}_m$ is more or less distant from that of
$s_m$. For instance, when $S_m$ has a finite Vapnik--Chervonenkis
dimension $V_m$ \cite{Vap1998,vcucrfep71} and $s\in S_m$, it has
been proven (see, e.g., \cite{Lug2002}) that
\[
\mathbb{E}[ \ell(s, \hat{s}_m ) ] \leq C \sqrt{\frac{V_m}{n}}
\]
for some numerical constant $C>0$. This is optimal without any
assumption on $P$, in the minimax sense: no estimator can have a
smaller prediction risk uniformly over all distributions $P$ such that
$s\in S_m$, up to the numerical factor $C$ \cite{DevLug1995}.

However, there exist favorable situations where much smaller prediction
errors (``fast rates'', up to $n^{-1}$ instead of $n^{-1/2}$) can be
obtained. A sufficient condition, the so-called ``margin condition'',
has been introduced by Mammen and Tsybakov \cite{MamTsy1999}. If, for
some $\varepsilon_0,C_0>0$ and $\alpha\geq1$,
%
%
\begin{equation} \label{eq.margin.global.suff1}
\forall\varepsilon\in(0,\varepsilon_0]\qquad\mathbb{P}\bigl(| 2
\eta(X) - 1 | \leq\varepsilon\bigr) \leq C_0 \varepsilon^{\alpha},
\end{equation}
if the Bayes predictor $s$ belongs to $S_m$, and if $S_m$ is a
VC-class of dimension $V_m$, then the prediction error of $\hat{s}_m$ is
smaller than $L(C_0,\varepsilon_0,\alpha) \ln(n) ( V_m / n )^{\kappa
/(2\kappa- 1)}$ in expectation, where $\kappa=
(1+\alpha)/\alpha$ and $L(C_0,\varepsilon_0,\alpha)>0$ only depends on
$C_0$, $\varepsilon_0$ and $\alpha$. Corresponding minimax lower bounds
\cite{MasNed2003} and other upper bounds can be obtained under other
complexity assumptions (e.g., Assumption (A2) of Tsybakov
\cite{Tsy2004}, involving bracketing entropy). In the extreme situation
where $\alpha= + \infty$, that is, for some $h>0$,
%
%
\begin{equation} \label{eq.margin.global.suff2}
\mathbb{P}\bigl(| 2 \eta(X) - 1 | \leq h\bigr) = 0,
\end{equation}
then the same result holds with $\kappa= 1$ and $L(h) \propto h^{-1}$.
More precisely, as proved in \cite{MasNed2003}
\[
\mathbb{E}[ \ell(s, \hat{s}_m ) ] \leq C \min\Biggl\{\biggl
(\frac
{V_m( 1+ \ln( n h^2 V_m^{-1} ) )}{n
h}\biggr),\sqrt{\frac{V}{n}}\Biggr\} .
\]

Following the approach of Koltchinskii \cite{Kol2006}, we will consider
the following generalization of the margin condition:
%
%
\begin{equation} \label{eq.margin.global}
\forall t \in S\qquad\ell(s, t ) \geq\varphi\bigl(\sqrt
{\operatorname{var}_P
\bigl(\gamma( t;\cdot) - \gamma( s;\cdot)\bigr)}\bigr),
\end{equation}
where $S$ is the set of predictors, and $\varphi$ is a convex
non-decreasing function on $[0,\infty)$ with $\varphi(0)=0$.
Indeed, the proofs of the above upper bounds on the prediction error of
$\hat{s}_m$ use only that (\ref{eq.margin.global.suff1}) implies
(\ref{eq.margin.global}) with $\varphi(x) = L(C_0,\varepsilon
_0,\alpha) x^{2\kappa}$ and $\kappa= (1+\alpha)/\alpha$, and that
(\ref{eq.margin.global.suff2}) implies (\ref{eq.margin.global}) with
$\varphi(x) = h x^2$. (See, e.g., Proposition 1 in \cite{Tsy2004}.)

All these results show that the empirical risk minimizer is
\textit{adaptive to the margin condition}, since it leads to an
optimal excess
risk under various assumptions on the complexity of $S_m$. However,
obtaining such rates of estimation requires knowledge of some $S_m$ to
which the Bayes predictor belongs, which is a strong assumption.

A less restrictive framework is the following. First, we do not assume
that $s\in S_m$.
Second, we do not assume that the margin condition
(\ref{eq.margin.global}) is satisfied for all $t \in S $, but only for
$t \in S_m $, which can be seen as a ``local'' margin condition:
%
\begin{equation} \label{eq.margin.local}
\forall t \in S_m\qquad\ell(s, t ) \geq\varphi_m
\bigl(\sqrt{\operatorname{var}_P \bigl(\gamma( t ; \cdot) -
\gamma( s; \cdot)\bigr)}\bigr),
\end{equation}
where $\varphi_m$ is a convex non-decreasing function on $[0,\infty)$
with $\varphi_m(0)=0 $. The fact that $\varphi_m$ can depend on $m$
allows situations where we are lucky to have a strong margin condition
for some small models but the global margin condition is loose. As
proven in Section \ref{sec.ex} (Proposition
\ref{le.local.vs.global.margin}), such situations certainly exist.

Note that when $\varphi_m(x)=h_m x^2$, (\ref{eq.margin.global}) and
(\ref{eq.margin.local}) can be traced back to mean--variance conditions
on $\gamma$ that were used in several papers for deriving convergence
rates of some minimum contrast estimators on some given model $S_m$
(see, e.g., \cite{BirMas1998} and references therein).

\subsection{Adaptive model selection} \label{sec.setting.mod-sel}
Assume now that we are not given a single model but a whole family
$( S_m )_{m \in\mathcal{M}_n}$. By empirical risk minimization, we
obtain a
family $( \hat{s}_m )_{m \in\mathcal{M}_n}$ of predictors, from
which we would like
to select some $\widehat{s}_{\widehat{m}}$ with a prediction error
$P\gamma( \widehat{s}_{\widehat{m}} )$ as small as possible.
The aim of such a \textit{model selection procedure} $((X_1,Y_1),
\ldots, (X_n,Y_n)) \mapsto\widehat{m}\in\mathcal{M}_n$ is to
satisfy an \textit{oracle
inequality} of the form
%
%
\begin{equation} \label{eq.oracle}
\ell(s, \widehat{s}_{\widehat{m}} ) \leq C \inf_{m \in\mathcal
{M}_n} \{ \ell(s, s_m ) + R_{m,n} \},
\end{equation}
where the leading constant $C \geq1$ should be close to one and the
remainder term $R_{m,n}$ should be close to $P\gamma( \hat{s}_m ) -
P\gamma( s_m )$. Typically, one proves that (\ref{eq.oracle})
holds either in expectation, or with high probability.

Assume for instance that $\varphi_m(x) = h_m x^2$ for some $h_m >0$ and
$S_m$ has a finite VC-dimension $V_m \geq1$. In view of the
aforementioned minimax lower bounds of \cite{MasNed2003}, one cannot
hope in general to prove an oracle inequality (\ref{eq.oracle}) with a
remainder $R_{m,n}$ smaller than
\[
\min\Biggl\{\frac{\ln(n) V_m}{n h_m},\sqrt{\frac{V_m}{n}}\Biggr\},
\]
where the $\ln(n)$ term may only be necessary for some VC classes $S_m$
(see \cite{MasNed2003}).

Then, \textit{adaptive model selection} occurs when $\widehat{m}$
satisfies an
oracle inequality (\ref{eq.oracle}) with $R_{m,n}$ of the order of this
minimax lower bound.
More generally, let $C_m$ be some complexity measure of $S_m$ (e.g.,
its VC-dimension, or the $\rho$ appearing in Tsybakov's assumption
\cite{Tsy2004}). Then, define $R_n(C_m,\varphi_m)$ as the minimax
prediction error over the set of distributions $P$ such that $s
\in S_m$ and the local margin condition (\ref{eq.margin.local}) is
satisfied in $S_m$ with $\varphi_m$, where $S_m$ has a complexity at
most $C_m$.
Massart and N\'ed\'elec \cite{MasNed2003} have proven tight upper and
lower bounds on $R_n(C_m,\varphi_m)$ with several complexity measures;
their results are stated with the margin condition
(\ref{eq.margin.global}), but they actually use its local version
(\ref{eq.margin.local}) only.

A margin adaptive model selection procedure should satisfy an oracle
inequality of the form
%
\begin{equation} \label{eq.oracle.margin}
\ell(s, \widehat{s}_{\widehat{m}} ) \leq C \inf_{m \in\mathcal
{M}_n} \{ \ell(s, s_m ) + R_n(C_m,\varphi_m) \}
\end{equation}
\textit{without using the knowledge of $C_m$ and $\varphi_m$.} We call
this property ``strong margin adaptivity'', to emphasize the fact that
this is more challenging than adaptivity to a margin condition that
holds uniformly over the models.

\subsection{Penalization} \label{sec.setting.pen}

We focus in particular in this paper on \textit{penalization}
procedures, which are defined as follows. Let $\operatorname{pen}\dvtx
\mathcal{M}_n
\mapsto[0,\infty)$ be a (data-dependent) function, and define
\[
\widehat{m}\in\arg\min_{m \in\mathcal{M}_n} \{ P_n \gamma(
\hat{s}_m ) + \operatorname{pen} (m) \}.
\]
Since our goal is to minimize the prediction error of $\hat{s}_m$, the
\textit{ideal penalty} would be
%
\begin{equation} \label{eq.penid}
\operatorname{pen}_{\mathrm{id}}(m) :=P \gamma( \hat{s}_m ) -
P_n \gamma( \hat{s}_m ),
\end{equation}
but it is unknown because it depends on the distribution $P$. A
classical way of designing a penalty is to estimate $\operatorname
{pen}_{\mathrm{id}}(m)$, or at
least a tight upper bound on it.

We consider in particular \textit{local complexity measures}
\cite{LugWeg2004,BarMenPhi2004,BarBouMen2005,Kol2006}, because they
estimate $\operatorname{pen}_{\mathrm{id}}$ tightly
enough to achieve fast estimation rates when the margin condition holds
true. See Section~\ref{sec.setting.loc-complex} for a detailed
definition of these penalties.

\subsection{Related results} \label{sec.setting.biblio}
There is a considerable wealth of literature on margin adaptivity in
the context of model selection as well as model aggregation.
Most of the papers consider the uniform margin condition, that is, when
$\varphi_m \equiv\varphi$.
Barron, Birg\'e and Massart \cite{BarBirMas1999} have proven oracle
inequalities for deterministic penalties under some mean--variance
condition on $\gamma$ close to (\ref{eq.margin.global}) with
$\varphi(x) = h x^2$.
Following a similar approach, margin adaptive oracle inequalities (with
more general $\varphi$) have been proven with localized random
penalties \cite{LugWeg2004,BarMenPhi2004,BarBouMen2005,Kol2006} and
with other penalties in a
particular framework~\cite{TsyvdG2005}.

Adaptivity to the margin has also been considered with a regularized
boosting method \cite{BlaLugVay2003}, the hold-out \cite{BlaMas2006}
and in a PAC-Bayes framework \cite{Aud2004PMA908}. Aggregation methods
have been studied in \cite{Tsy2004} and \cite{Lec2007c}. Notice also
that a completely different approach is possible: estimate first the
regression function $\eta$ (possibly through model selection), then use
a plug-in classifier; this works provided $\eta$ is smooth enough
\cite{AudTsy2007}.

It is quite unclear whether any of these results can be extended to
strong margin adaptivity (actually, we will prove that this needs
additional restrictions in general). To our knowledge, the only results
allowing $\varphi_m$ to depend on $m$ can be found in \cite{Kol2006}.
First, when the models are nested, a comparison method based on local
Rademacher complexities attains strong margin adaptivity, assuming that
$s\in\bigcup_{m \in\mathcal{M}_n} S_m$ (Theorem 7; and it is quite unclear
whether this still holds without the latter assumption). Second, a
penalization method based on local Rademacher complexities has the same
property in the general case, but it uses the knowledge of
$( \varphi_m )_{m \in\mathcal{M}_n}$ (Theorems 6 and 11).

Our claim is that when $\varphi_m$ does strongly depend on $m$, it is
crucial to take it into account to choose the best model in $\mathcal
{M}_n$. And
such situations occur, as proven by our Proposition
\ref{le.local.vs.global.margin} in Section \ref{sec.ex}. But assuming
either $s\in\bigcup_{m \in\mathcal{M}_n} S_m$ or that $\varphi_m$
is known is
not realistic. Our goal is to investigate the kind of results that can
be obtained with \textit{completely data-driven} procedures; in
particular, when $s\notin\bigcup_{m \in\mathcal{M}_n} S_m$.

\subsection{Our results}

In this paper, we aim at understanding when strong margin adaptivity
can be obtained for data-dependent model selection procedures. Notice
that we do not restrict ourselves to the classification setting. We
consider a much more general framework (as in \cite{Kol2006}), which is
described in Section \ref{sec.gal-set}. We prove two kinds of results.
First, when models are nested, we show that some penalization methods
are strongly margin adaptive (Theorem \ref{th.nested}). In particular,
this result holds for the local Rademacher complexities (Corollary~\ref{cor.nested.locRad}). Compared to previous results (in particular
the ones of \cite{Kol2006}), our main advance is that our penalties do
not require the knowledge of~$( \varphi_m )_{m \in\mathcal{M}_n}$,
and we do not
assume that the Bayes predictor belongs to any of the models.

Our second result probes the limits of strong margin adaptivity,
without the nested assumption. A family of models exists such that, for
every sample size $n$ and every (model) selection procedure $\widehat
{m}$, a
distribution $P$ exists for which $\widehat{m}$ fails to be strongly margin
adaptive with a positive probability (Theorem \ref{th.non-nested.new}).
Hence, the previous positive results (Theorem \ref{th.nested} and
Corollary~\ref{cor.nested.locRad}) cannot be extended outside of the
nested case for a general distribution $P$.

Where is the boundary between these two extremes? Obviously, the nested
assumption is not necessary. For instance, when the global margin
assumption is indeed tight ($\varphi= \varphi_m$ for every $m \in
\mathcal{M}_n$),
margin adaptivity can be obtained in several ways, as mentioned in
Section \ref{sec.setting.biblio}. We sketch in Section
\ref{sec.oracle.gal} some situations where strong margin adaptivity is
possible. More precisely, we state a general oracle inequality (Theorem
\ref{th.gal.oracle}), valid for any family of models and any
distribution $P$. We then discuss assumptions under which its remainder
term is small enough to imply strong margin adaptivity.

This paper is organized as follows. We describe the general setting in
Section \ref{sec.gal-set}. We consider in Section \ref{sec.nested} the
nested case, in which strong margin adaptivity holds. Negative results
(i.e., lower bounds on the prediction error of a general model
selection procedure) are stated in Section~\ref{sec.non-nested}. The
line between these two situations is sketched in Section
\ref{sec.oracle.gal}. We discuss our results in Section~\ref{sec.discussion}. All the proofs are given in Section \ref{sec.proofs}.

\section{The general empirical minimization framework} \label{sec.gal-set}

Although our main motivation comes from the classification problem, it
turns out that all our results can be proven in the general setting of
empirical minimization. As explained below, this setting includes
binary classification with the 0--1 loss, bounded regression and
several other frameworks. In the rest of the paper, we will use the
following general notation, in order to emphasize the generality of our
results.

We observe independent realizations $\xi_1, \ldots, \xi_n \in\Xi$
of a
random variable with distribution $P$, and we are given a set $\mathcal
{F}$ of
measurable functions $\Xi\mapsto[0,1]$. Our goal is to build some
(data-dependent) $f$ such that its expectation $P(f) :=\mathbb{E}_{\xi
\sim P} [ f(\xi) ]$ is as small as possible. For the sake of
simplicity, we assume that there is a minimizer $f^{\star}$ of $P(f)$ over
$\mathcal{F}$.

This includes the prediction framework, in which $\Xi= \mathcal
{X}\times\mathcal{Y}$,
$\xi_i = (X_i,Y_i)$,
\[
\mathcal{F}:=\{ \xi\mapsto\gamma( t ; \xi) \mbox{ s.t. }t \in S
\},
\]
where $\gamma:S\times\Xi\mapsto[0,1]$ is any contrast function. Then,
$f^{\star}$ is equal to $\gamma( s; \cdot)$, where $s$
is the
Bayes predictor.
In the binary classification framework, $\mathcal{Y}= \{ 0,1 \}$ and
we can
take the 0--1 contrast $\gamma(t;(x,y)) = \mathbh{1}_{t(x) \neq y},$ for
instance.
We then recover the setting described in Section \ref{sec.setting}. In
the bounded regression framework, assuming that $\mathcal{Y}= [0,1]$,
we can
take the least-squares contrast,
\[
\gamma(t;(x,y)) = \bigl( t(x) - y \bigr)^2.
\]
Many other contrast functions $\gamma$ can be considered, provided that
they take their values in $[0,1]$.
Notice the one-to-one correspondence between predictors $t$ and
functions $\overline{f}_t :=\gamma( t;\cdot)$ in the prediction
framework.

The \textit{empirical minimizer} over $\mathcal{F}_m \subset\mathcal
{F}$ (called a model)
can then be defined as
\[
\widehat{f}_m \in\arg\min_{f \in\mathcal{F}_m} P_n (f).
\]
We expect that its expectation $P(\widehat{f}_m)$ is close to that of
$f_m \in
\arg\min_{f \in\mathcal{F}_m} P(f) $, assuming that such a
minimizer exists. In
the prediction framework, defining $\mathcal{F}_m :=\{ \overline{f}_t
\mbox{ s.t. }t \in S_m \}$, we have $\widehat{f}_m = \overline
{f}_{\hat{s}_m}$ and $f_m =
\overline{f}_{s_m}$.

We can now write the global margin condition as follows:
%
\begin{equation} \label{eq.margin.global.gal}
\forall f \in\mathcal{F}\qquad P(f - f^{\star}) \geq\varphi
\bigl( \sqrt{\operatorname{var}_P (f - f^{\star})}\bigr ),
\end{equation}
where $\varphi$ is a
convex non-decreasing function on $[0,\infty)$ with $\varphi(0)=0$.
Similarly, the local margin condition is
%
\begin{equation} \label{eq.margin.local.gal}
\forall f \in\mathcal{F}_m\qquad P(f - f^{\star}) \geq\varphi_m
\bigl( \sqrt{\operatorname{var}_P (f - f^{\star})} \bigr).
\end{equation}
Notice that most of the upper and lower bounds on the risk under the
margin condition given in the introduction stay valid in the general
empirical minimization framework, at least when $\varphi_m(x) = (h_m
x^{2})^{\kappa_m}$ for some $h_m>0$ and $\kappa_m \geq1$ (see, e.g.,
\cite{MasNed2003} and \cite{Kol2006}).
Assume that $\mathcal{F}_m$ is a~VC-type class of dimension $V_m$. If
$\varphi_m(x) = h_m x^2$,
\[
\mathbb{E}[ P( \widehat{f}_m - f^{\star}) ] \leq2 P(f_m - f^{\star
}) + C \min\Biggl\{\biggl( \frac{\ln(n) V_m}{n h_m} \biggr),\sqrt{\frac
{V_m}{n}}\Biggr\}
\]
for some numerical constant $C>0$. If $\varphi_m(x) = ( h_m x^2
)^{\kappa_m}$ for some $h_m>0$ and $\kappa_m \geq1$,
\[
\mathbb{E}[ P( \widehat{f}_m - f^{\star}) ] \leq2 P(f_m - f^{\star
}) + C \min\Biggl\{\biggl[ L(h_m,\kappa_m) \ln(n) \biggl( \frac{ V_m}{n h_m}
\biggr)^{\kappa_m/(2 \kappa_m - 1)} \biggr],\sqrt{\frac{V_m}{n}}\Biggr
\}
\]
for some constants $C,L(h_m,\kappa_m)>0$.

Given a collection $( \mathcal{F}_m )_{m \in\mathcal{M}_n}$ of
models, we are looking for
a model selection procedure $(\xi_1, \ldots, \xi_n) \mapsto\widehat
{m}\in
\mathcal{M}_n$
satisfying an \textit{oracle inequality} of the form
%
\begin{equation} \label{eq.oracle.gal}
P ( \widehat{f}_{\widehat{m}} - f^{\star}) \leq C \inf_{m \in
\mathcal{M}_n} \{ P ( f_m - f^{\star}) + R_{m,n} \},
\end{equation}
with a leading constant $C$ close to 1 and a remainder term $R_{m,n}$
as small as possible. Similarly to (\ref{eq.oracle.margin}), we define
a strongly margin-adaptive procedure as any $\widehat{m}$ such that
(\ref{eq.oracle.gal}) holds with some numerical constant $C$, and
$R_{m,n}$ of the order of the minimax risk $R_n(C_m,\varphi_m)$, where
$C_m$ is some complexity measure of $\mathcal{F}_m$.

Defining penalization methods as
%
\begin{equation} \label{eq.pen}
\widehat{m}\in\arg\min_{m \in\mathcal{M}_n} \{ P_n ( \widehat
{f}_m ) + \operatorname{pen}(m) \}
\end{equation}
for some data-dependent $\operatorname{pen}\dvtx \mathcal{M}_n \mapsto
\mathbb{R}$, the ideal penalty is
$\operatorname{pen}_{\mathrm{id}}(m) :=(P-P_n) ( \widehat{f}_m )$.

\section{Margin-adaptive model selection for nested models} \label{sec.nested}
%
\subsection{General result}
Our first result is a sufficient condition for penalization procedures
to attain strong margin adaptivity when the models are nested (Theorem
\ref{th.nested}). Since this condition is satisfied by local Rademacher
complexities, this leads to a data-driven margin-adaptive penalization
procedure (Corollary \ref{cor.nested.locRad}).

\begin{theorem} \label{th.nested}
Fix $( \mathcal{F}_m )_{m \in\mathcal{M}_n}$ and $( \varphi_m )_{m
\in\mathcal{M}_n}$ such that
the local margin conditions (\ref{eq.margin.local.gal}) hold. Let
$( t_m )_{m \in\mathcal{M}_n}$ be a sequence of positive reals that is
non-decreasing (with respect to the inclusion ordering on
$\mathcal{F}_m$). Assume that some constants $c,\eta\in
(0,1)$ and
$C_1,C_2 \geq0$ exist such that the following holds:
\begin{itemize}
\item The models $( \mathcal{F}_m )_{m \in\mathcal{M}_n}$ are nested.
\item Lower bounds on the penalty: with probability at least $1 - \eta
$, for every $m,m^{\prime} \in\mathcal{M}_n$,
\begin{eqnarray} \label{eq.pen.minor.nested}
&&\hspace*{32.3pt}(1-c) \operatorname{pen}(m) \geq(P - P_n) ( \widehat{f}_m - f_m ) +
\frac{t_m}{n}
\geq0,
\\ \label{eq.nested.min-pen}
&&\mathcal{F}_{m^{\prime}} \subset\mathcal{F}_m \Rightarrow c
\operatorname{pen}(m) \geq v(m)
- C_1 v(m^{\prime}) - C_2 P( f_{m^{\prime}} - f^{\star}),\nonumber\\
[-8pt]\\ [-8pt]
&&\qquad\mbox{where } v(m) :=\sqrt{ \frac{2t_m}{n} \operatorname{var}_P
(f_m - f^{\star})}.\nonumber
\end{eqnarray}
\end{itemize}
Then, if $\widehat{m}$ is defined by (\ref{eq.pen}), with probability
at least
$1 - \eta- 2 \sum_{m \in\mathcal{M}_n} \mathrm{e}^{-t_m}$, we have for every
$\varepsilon
\in(0,1)$
%
\begin{eqnarray} \label{eq.nested}
\hspace*{-10pt}P ( \widehat{f}_{\widehat{m}} - f^{\star}) &\leq&\frac{1}{1 -
\varepsilon} \inf
_{m \in\mathcal{M}_n} \Biggl\{ (1 + \varepsilon+ C_2 + \varepsilon
C_1) P ( f_{m}
- f^{\star}) + \operatorname{pen}(m)\nonumber  \\ [-8pt]\\ [-8pt]
 &&\hphantom{\frac{1}{1 -
\varepsilon} \inf
_{m \in\mathcal{M}_n} \Biggl\{}{}+ ( 1 + \max\{ 1,C_1 \} ) \min\Biggl\{\varphi_m^{\star} \Biggl(
\sqrt{ \frac{2 t_m}{ \varepsilon^2 n}} \Biggr),\sqrt{\frac{2
t_m}{n}}\Biggr\} +
\frac{t_m}{3n} \Biggr\} ,\nonumber
\end{eqnarray}
where $\varphi_m^{\star}(x) :=\sup_{y \geq0} \{ x y - \varphi_m(y)
\} $ is the convex conjugate of $\varphi_m$.
\end{theorem}

Theorem \ref{th.nested} is proved in Section \ref{sec.proofs.oracle}.

\begin{remark}
\begin{enumerate}

\item If $\operatorname{pen}(m)$ is of the right order, that is, not
much larger than
$\mathbb{E}[ \operatorname{pen}_{\mathrm{id}}(m) ]$, then Theorem
\ref{th.nested} is a strong margin
adaptivity result. Indeed, assuming that $\varphi_m(x) = ( h_m x^2
)^{\kappa_m}$, the remainder term is not too large, since
\[
\varphi_m^{\star}(x) = L(h_m,\kappa_m) x^{2 \kappa_m /
(2 \kappa_m - 1)}
\]
for some positive constant $L(h_m,\kappa_m)$. Hence, choosing
$\varepsilon= 1/2,$ for instance, we can rewrite (\ref{eq.nested}) as
\[
P ( \widehat{f}_{\widehat{m}} - f^{\star})\leq L(C_1,C_2)
\inf_{m \in\mathcal{M}_n}
\biggl\{ P ( f_{m} - f^{\star}) + \operatorname{pen}(m) + L(h_m,\kappa_m)
\biggl( \frac{t_m}{n} \biggr)^{\kappa_m/(2\kappa_m - 1)} \biggr\}
\]
for some positive constants $L(C_1,C_2)$ and $L(h_m,\kappa_m)$.
When $\varphi_m$ is a general convex function, minimax estimation rates
are no longer available, so that we do not know whether the remainder
term in (\ref{eq.nested}) is of the right order. Nevertheless, no
better risk bound is known, even for a single model to which $s$
belongs.

\item In the case that the $\varphi_m$ are known, methods involving
local Rademacher complexities and $( \varphi_m )_{m \in\mathcal
{M}_n}$ satisfy
oracle inequalities similar to (\ref{eq.nested}) (see Theorems 6 and 11
in \cite{Kol2006}). On the contrary, the $\varphi_m$ are not assumed to
be known in Theorem \ref{th.nested}, and conditions
(\ref{eq.pen.minor.nested}) and~(\ref{eq.nested.min-pen}) are satisfied
by completely data-dependent penalties, as shown in Section
\ref{sec.setting.loc-complex}. Also, Theorem 7 of \cite{Kol2006} shows
that adaptivity is possible using a comparison method, provided that
$f^{\star}$ belongs to one of the models. However, it is not clear whether
this comparison method achieves the optimal bias--variance trade-off in
the general case, as in Theorem \ref{th.nested}.
\label{rk.nested.tm=t}
\end{enumerate}
\end{remark}

\subsection{Local Rademacher complexities} \label{sec.setting.loc-complex}

Although Theorem \ref{th.nested} applies to any penalization procedure
satisfying assumptions (\ref{eq.pen.minor.nested}) and
(\ref{eq.nested.min-pen}), we now focus on methods based on local
Rademacher complexities. Let us define precisely these complexities. We
mainly use the notation of \cite{Kol2006}:
\begin{itemize}
\item for every $\delta>0$, the $\delta$ minimal set of $\mathcal
{F}_m$ w.r.t.
the distribution $P$ is
\[
\mathcal{F}_{m,P}(\delta) :=\Bigl\{ f \in\mathcal{F}_m \mbox{ s.t.
}P(f) - \inf_{g \in\mathcal{F}_m} P(g) \leq\delta \Bigr\},
\]
\item the $L^2(P)$ diameter of the $\delta$ minimal set of $\mathcal
{F}_m$ is
\[
D_{P}^2(\mathcal{F}_m ; \delta) = \sup_{f,g \in\mathcal{F}_{m,P}
(\delta)} P \bigl((f-g)^2 \bigr),
\]
\item the expected modulus of continuity of $(P-P_n)$ over $\mathcal
{F}_m$ is
\[
\phi_n ( \mathcal{F}_m ; P ; \delta) = \mathbb{E}\sup_{ f,g \in
\mathcal{F}_{m,P} (\delta
)} | (P_n - P) (f - g) |.
\]
\end{itemize}
We then define
\[
U_n (\mathcal{F}_m ; \delta; t) :=\overline{K} \Biggl( \phi_n ( \mathcal
{F}_m ; P
; \delta) + D_P(\mathcal{F}_m ; \delta) \sqrt{ \frac{t}{n} } +
\frac
{t}{n}\Biggr ),
\]
where $\overline{K}>0$ is a numerical constant (to be chosen later).
The (ideal) local complexity $\overline{\delta}_n (\mathcal{F}_m ;
t)$ is
(roughly) the smallest positive fixed point of $r \mapsto U_n (\mathcal
{F}_m ; r
; t)$. More precisely,
%
\begin{equation} \label{eq.loc-comp.ideal}
\overline{\delta}_n (\mathcal{F}_m ; t) :=\inf\biggl\{ \delta> 0 \mbox
{ s.t. } \sup_{\sigma\geq\delta} \biggl\{ \frac{U_n (\mathcal{F}_m ;
\sigma; t)}{\sigma} \biggr\} \leq\frac{1}{2q} \biggr\},
\end{equation}
where $q >1$ is a numerical constant.

Two important points, which follow from Theorems 1 and 3 of
Koltchinskii \cite{Kol2006}, are that:
\begin{enumerate}
\item$\overline{\delta}_n(\mathcal{F}_m;t)$ is large enough to satisfy
assumption (\ref{eq.pen.minor.nested}) with a probability at least $1 -
\log_q ( n/t ) \mathrm{e}^{-t}$ for each model $m \in\mathcal{M}_n$.
\item There is a
completely data-dependent $\hat{\delta}_n(\mathcal{F}_m;t)$ such that
\[
\forall m \in\mathcal{M}_n\qquad \mathbb{P}\bigl( \hat{\delta
}_n(\mathcal{F}_m;t) \geq\overline
{\delta}_n (\mathcal{F}_m ; t) \bigr) \geq1 - 5 \ln_q\biggl( \frac{n}{t} \biggr)
\mathrm{e}^{-t} .
\]
This data-dependent $\hat{\delta}_n(\mathcal{F}_m;t)$ is a resampling
estimate of $\overline{\delta}_n (\mathcal{F}_m ; t)$, called the ``local
Rademacher complexity''.
\end{enumerate}

Before stating the main result of this section, let us recall the
definition of $\hat{\delta}_n(\mathcal{F}_m;t)$, as in~\cite{Kol2006}.
We
need the following additional notation:
\begin{itemize}
\item for every $\delta>0$, the empirical $\delta$ minimal set of
$\mathcal{F}_m$ is
\[
\widehat{\mathcal{F}}_{n,m}(\delta)
:=\Bigl\{ f \in\mathcal{F}_m \mbox{ s.t. }P_n(f) - \inf_{g \in
\mathcal{F}_m} P_n(g) \leq\delta \Bigr\} = \mathcal{F}_{m,P_n}(\delta),
\]
\item the empirical $L^2(P)$ diameter of the empirical $\delta$ minimal
set of $\mathcal{F}_m$ is
\[
\widehat{D}_{n}(\mathcal{F}_m ; \delta) = \sup_{f,g \in\widehat
{\mathcal{F}}_{n,m}(\delta)} P_n\bigl( (f-g)^2 \bigr),
\]
\item the modulus of continuity of the Rademacher process $ f \mapsto
n^{-1} \sum_{i=1}^n \varepsilon_i f(\xi_i)$ over $\mathcal{F}_m$, where
$\varepsilon_1, \ldots, \varepsilon_n$ are i.i.d. Rademacher random
variables (i.e., $\varepsilon_i$ takes the values $+1$ and $-1$ with
probability $1/2$ each):
\[
\widehat{\phi}_n ( \mathcal{F}_m ; \delta) = \sup_{ f,g \in
\widehat{\mathcal{F}
}_{n,m} (\delta)} \Biggl| \frac{1}{n} \sum_{i=1}^n \varepsilon_i \bigl(f(\xi
_i) - g(\xi_i)\bigr) \Biggr|.
\]
\end{itemize}
Defining
\[
\widehat{U}_n (\mathcal{F}_m ; \delta; t) :=\widehat{K} \Biggl( \widehat
{\phi}_n ( \mathcal{F}_m ; P ; \hat{c} \delta) + \widehat
{D}_{n}(\mathcal{F}_m; \hat{c} \delta) \sqrt{ \frac{t}{n} } + \frac{t}{n} \Biggr)
\]
(where $\widehat{K},\hat{c}>0$ are numerical constants, to be
chosen later), the \textit{local Rademacher complexity}
$\hat{\delta}_n (\mathcal{F}_m ; t)$ is (roughly) the smallest positive
fixed point of $r \mapsto\widehat{U}_n (\mathcal{F}_m ; r ; t)$.
More precisely,
%
\begin{equation} \label{eq.locRad}
\hat{\delta}_n (\mathcal{F}_m ; t) :=\inf\biggl\{ \delta> 0 \mbox
{ s.t. } \sup_{\sigma\geq\delta} \biggl\{ \frac{\widehat{U}_n (\mathcal
{F}_m ; \sigma; t)}{\sigma} \biggr\} \leq\frac{1}{2q} \biggr\},
\end{equation}
where $q >1$ is a numerical constant.

\begin{corollary}[(Strong margin adaptivity for local Rademacher
complexities)] \label{cor.nested.locRad}
There exist numerical constants $\overline{K}>0$ and $q>1$ such that
the following holds. Let $t>0$.
Assume that a numerical constant $L>0$ exists and an event of
probability at least $1 - L \log_q (n/t) \operatorname{Card}(\mathcal
{M}_n) \mathrm{e}^{-t}$ exists on
which
\begin{equation} \label{eq.nested.min-pen.locRad} \forall m \in
\mathcal{M}_n\qquad \operatorname{pen}(m) \geq\frac{7}{2} \overline{\delta
}_n(\mathcal{F}_m ; t),
\end{equation}
where $\overline{\delta}_n(\mathcal{F}_m ; t)$ is defined by
(\ref{eq.loc-comp.ideal}) (and depends on both $\overline{K}$
and $q$).
Assume moreover that the models $( \mathcal{F}_m )_{m \in\mathcal
{M}_n}$ are nested and
\[
\widehat{m}\in\arg\min_{m \in\mathcal{M}_n} \{ P_n ( \widehat
{f}_m ) + \operatorname{pen}(m) \}.
\]
Then, an event of probability at least $1 - [ 2 + (L+1) \log_q ( \frac
{n}{t} ) ] \operatorname{Card}(\mathcal{M}_n) \mathrm{e}^{-t}$ exists on
which, for every
$\varepsilon\in(0,1)$,
%
\begin{eqnarray} \label{eq.nested.locRad}
P ( \widehat{f}_{\widehat{m}} - f^{\star}) &\leq&\frac{1}{1 -
\varepsilon} \inf_{m \in\mathcal{M}_n} \Biggl\{ \biggl( 1 + \frac{2}{\overline{K} q} +
\varepsilon\bigl(1+ \sqrt{2} \bigr) \biggr) P ( f_{m} - f^{\star}) +
\operatorname{pen}(m)\nonumber\\ [-8pt]\\ [-8pt]
 &&\hphantom{\frac{1}{1 -\varepsilon} \inf
_{m \in\mathcal{M}_n} \Biggl\{}{}+ \bigl(1 + \sqrt{2}\bigr) \min\Biggl\{ \varphi_m^{\star} \Biggl( \sqrt{ \frac
{2t}{ \varepsilon^2 n}} \Biggr) , \sqrt{ \frac{2t}{n} } \Biggr\} +
\frac{t}{3n} \Biggr\}.\nonumber
\end{eqnarray}
In particular, this holds when $\operatorname{pen}(m) = \frac{7}{2}
\hat{\delta}_n(\mathcal{F}_m;t) $, provided that $\widehat
{K},\hat{c}>0$
are larger than some constants depending only on $\overline{K},q$.
\end{corollary}

Corollary \ref{cor.nested.locRad} is proved in Section
\ref{sec.proofs.oracle}.
\begin{remark}
One can always enlarge the constants $\overline{K}$ and $q$, making the
leading constant of the oracle inequality (\ref{eq.nested.locRad})
closer to one, at the price of enlarging $\overline{\delta
}_n(\mathcal{F}_m ;
t)$ (hence $\operatorname{pen}(m)$ or $\hat{\delta}_n(\mathcal
{F}_m;t) $). We do not know
whether it is possible to make the leading constant closer to one
without changing the penalization procedure itself.
\end{remark}

As we show in Section \ref{sec.ex}, there are distributions $P$ and
collections of models $( \mathcal{F}_m )_{m \in\mathcal{M}_n}$ such that
(\ref{eq.nested.locRad}) is a strong improvement over the ``uniform
margin'' case, in terms of prediction error. It seems reasonable to
expect that this happens in a significant number of practical
situations.

In Section \ref{sec.oracle.gal}, we state a more general result (from
which Theorem \ref{th.nested} is a corollary) that suggests why it is
more difficult to prove Corollary \ref{cor.nested.locRad} when
$\varphi_m$ really depends on~$m$. This general result is also useful
to understand how the nestedness assumption might be relaxed in
Theorem~\ref{th.nested} and Corollary \ref{cor.nested.locRad}.

The reason why Corollary \ref{cor.nested.locRad} implies strong margin
adaptivity is that the local Rademacher complexities are not too large
when the local margin condition is satisfied, together with a
complexity assumption on $\mathcal{F}_m$. Indeed, there exists a
distribution-dependent $\widetilde{\delta}_n (\mathcal{F}_m ; t)$
(defined as
$\overline{\delta}_n(\mathcal{F}_m;t)$ with $U_n (\mathcal{F}_m ;
\delta; t)$
replaced by
$K_1 U_n (\mathcal{F}_m ; K_2 \delta; t)$ for some numerical constants
$K_1,K_2>0$, related to $\widehat{K}$ and $\hat{c}$) such~that
\[
\forall m \in\mathcal{M}_n \qquad \mathbb{P}\bigl( \widetilde{\delta
}_n (\mathcal{F}_m ; t)
\geq\hat{\delta}_n(\mathcal{F}_m;t) \geq\overline{\delta}_n
(\mathcal{F}_m ;
t) \bigr)
\geq1 - 5 \log_q\biggl( \frac{n}{t} \biggr) \mathrm{e}^{-t} .
\]
(See Theorem 3 of \cite{Kol2006}.) This leads to several upper bounds
on $\hat{\delta}_n(\mathcal{F}_m;t)$ under the local margin condition
(\ref{eq.margin.local.gal}), by combining Lemma 5 of \cite{Kol2006}
with the examples of its Section 2.5. For instance, in the binary
classification case, when $\mathcal{F}_m$ is the class of 0--1 loss functions
associated with a VC-class $S_m$ of dimension $V_m$, such that the
margin condition (\ref{eq.margin.local.gal}) holds with
$\varphi_m(x)=h_m x^2$, we have for every $t > 0$ and $\varepsilon\in
(0,1]$,
%
\begin{equation} \label{eq.maj.locRad.VC-margin}
\overline{\delta}_n (\mathcal{F}_m ; t) \leq\varepsilon P( f_m -
f^{\star}) +
\frac{K_3}{n h_m} \biggl[ \varepsilon^{-1} t + \varepsilon^{-2} V_m \ln \biggl(
\frac{n \varepsilon^2 h_m} {K_4 V_m} \biggr) \biggr],
\end{equation}
where $K_3$ and $K_4$ depend only on $\overline{K}$. (Similar upper
bounds hold under several other complexity assumptions on the models
$\mathcal{F}_m$, see \cite{Kol2006}.) In particular, when each model
$S_m$ is a
VC-class of dimension $V_m$, $\varphi_m(x)=h_m x^2$, $\operatorname
{pen}(m) =
\frac{7}{2} \hat{\delta}_n
( \mathcal{F}_m ; t )$ and $t= \ln(\operatorname{Card}(\mathcal
{M}_n)) + 3 \ln(n)$,
(\ref{eq.nested.locRad}) implies that
\[
P ( \widehat{f}_{\widehat{m}} - f^{\star})
\leq C \inf_{m \in\mathcal{M}_n} \biggl\{ P ( f_{m} - f^{\star}) + \frac
{\ln( \operatorname{Card}(\mathcal{M}_n) ) + \ln(n) + V_m \ln( \mathrm{e}n
h_m/V_m )}{n h_m} \biggr\}
\]
with probability at least $1 - K n^{-2}$, for some numerical constants
$C,K>0$. Up to some $\ln(n)$ factor, this is a strong margin-adaptive
model selection result, provided that $\operatorname{Card}(\mathcal
{M}_n)$ is smaller than
some power of $n$. Notice that the $\ln(n)$ factor is sometimes
necessary (as shown by \cite{MasNed2003}), meaning that this upper
bound is then optimal.

\section{Lower bound for some non-nested models} \label{sec.non-nested}
In this section, we investigate the assumption in Theorem
\ref{th.nested} that the models $\mathcal{F}_m$ are nested.
To this aim, let us consider the case where models are singletons
$\mathcal{F}_m
= \{ f_m \}$. Then, any estimator $\widehat{f}_m \in\mathcal{F}_m$
is deterministic and
equal to $f_m$, so that model selection amounts to selecting among a
family $\{ f_m \mbox{ s.t. }m \in\mathcal{M}_n \}$ of functions.
Theorem \ref{th.non-nested.new} below shows that no selection procedure
can be strongly margin-adaptive in general.
\begin{theorem} \label{th.non-nested.new}
Let $\gamma$ be the 0--1 loss and $\mathcal{F}^{0-1}:=
\{ \gamma( u;\cdot) \mbox{ s.t. }u\dvtx \mathcal{X}\mapsto\{ 0,1 \}
\mbox{ is measurable} \}$ be the associated loss function class.
If $\operatorname{Card}(\mathcal{X}) \geq2$, two functions $f_0, f_1
\in\mathcal{F}^{0-1}$ and
absolute constants $C_3,C_4>0$ exist such that the following holds.
For every integer $n \geq2$ and $\widehat{m}$ a selection procedure
(that is, a function $( \mathcal{X}\times\mathcal{Y})^n
\mapsto
\mathcal{M}=\{ 0,1 \} $), a distribution $P$ exists such that
\begin{equation} \label{eq.non-nested.minimax.prob}
\mathbb{P}\biggl( P ( f_{\widehat{m}} - f^{\star}) \geq\frac{C_4 \sqrt
{n}}{\ln(n)} \min
_{m \in\{ 0,1 \}} \biggl\{ P ( f_m - f^{\star}) + \overline{v}(m) + \frac
{\ln (n)}{n h_m} \biggr\} \biggr) \geq C_3
\end{equation}
and
\begin{equation}\label{eq.non-nested.minimax.esp}
\mathbb{E}[ P ( f_{\widehat{m}} - f^{\star}) ] \geq
\frac{C_3 C_4
\sqrt{n}}{\ln(n)} \min_{m \in\{ 0,1 \}} \biggl\{ P ( f_m - f^{\star}) +
\overline{v}(m) + \frac{\ln(n)}{n h_m} \biggr\},
\end{equation}
where $\forall m \in\{ 0,1 \}$
\[
\overline{v}(m) :=\sqrt{\frac{2\ln(n)}{n} \operatorname{var}_P (f_m - f^{\star}) }\quad\mbox{and}\quad h_m
:=\frac{P(f_m -f^{\star})}{\operatorname{var}_P (f_m - f^{\star})} .
\]
\end{theorem}

Theorem \ref{th.non-nested.new} is proved in Section
\ref{sec.proofs.lower}.
A straightforward corollary of Theorem \ref{th.non-nested.new} is that
in the classification setting with the 0--1 loss, \textit{strong
margin-adaptive model selection is not always possible when the models
are not nested}.
Indeed, when $\mathcal{F}_m = \{ f_m \}$ for every $m \in\mathcal
{M}_n = \{ 0,1 \}$,~(\ref{eq.non-nested.minimax.prob}) shows that for any model selection
procedure $\widehat{m}$, some distribution $P$ exists such that
results like
Theorem \ref{th.nested} or Corollary \ref{cor.nested.locRad} do not
hold if $t_m = \ln(n)$ for every $m $.
\begin{remark} \label{rk.non-nested.minimax.random}
\begin{enumerate}
\item Theorem \ref{th.non-nested.new} and its corollary for model
selection also hold for randomized rules $\widehat{m}\dvtx ( \mathcal
{X}\times\mathcal{Y})^n
\mapsto[0,1]$ (where the value of $\widehat{m}((X_i, Y_i)_{1 \leq i
\leq n})$ is
the probability assigned to the choice of $f_1$). Hence, aggregating
models instead of selecting one does not modify the conclusion of
Theorem \ref{th.non-nested.new}.
\item The most reasonable selection procedure among two functions $f_0$
and $f_1$ (or two models $\{ f_0 \}$ and $\{ f_1 \}$) clearly is
empirical minimization. The proof of Theorem \ref{th.non-nested.new}
yields explicitly some distribution $P$, called $P_1 $, such that
(\ref{eq.non-nested.minimax.prob}) and
(\ref{eq.non-nested.minimax.esp}) hold for empirical minimization.
Note that when models are singletons, most penalization procedures
coincide with empirical minimization, for instance, when $\operatorname
{pen}(m)$ is
proportional to the local Rademacher complexity $\hat{\delta}_n
(\mathcal{F}_m ; t)$, or to the ideal penalty $\operatorname
{pen}_{\mathrm{id}}(m) = (P - P_n)(\widehat{f}_m -
f_m)$, its expectation or some quantile of $\operatorname
{pen}_{\mathrm{id}}(m)$.
\item Theorem \ref{th.non-nested.new} focuses on margin adaptivity with
$\varphi_m(x) = h_m x^2$, whereas the margin condition is also
satisfied with other functions $\varphi_m$. This is both for simplicity
reasons and because this choice emphasizes that one could hope for
learning rates of order $1 / (n h_m)$ if strong margin adaptivity were
possible. The meaning of Theorem \ref{th.non-nested.new} is then mainly
that one cannot guarantee to learn at a rate better than $1/\sqrt{n}$,
whereas for some model, the excess loss and $1/(n h_m)$ both are of
order $1/n $.
\item The counterexample given in the proof of Theorem
\ref{th.non-nested.new} is highly non-asymptotic, since the
distribution $P$ strongly depends on $n$. If $P$ and $f_0, f_1$ were
fixed, it is well known that empirical minimization leads to asymptotic
optimality, because $( f_m )_{m\in\{ 0,1 \}}$ is finite and fixed
when $n$ grows.
This illustrates a significant difference between the asymptotic and
non-asymptotic frameworks. Another example of such a difference occurs
when the number of candidate functions (or models) is infinite, or
grows to infinity with the sample size, see~(iv) in Proposition
\ref{le.local.vs.global.margin} in Section \ref{sec.ex}.
\end{enumerate}
\end{remark}

With Theorem \ref{th.nested}, we have proven a strong margin adaptivity
result for nested models, which holds true when the penalty is built
upon local Rademacher complexities. Therefore, adaptive model selection
is attainable for nested models, whatever the distribution of the data.
On the other hand, Theorem \ref{th.non-nested.new} gives a simple
example where no model selection procedure can satisfy an oracle
inequality (\ref{eq.oracle.gal}) with a leading constant smaller than
$C_4 \sqrt{n} / (\ln(n)) $.

Looking carefully at the selection problems considered in the proof of
Theorem \ref{th.non-nested.new}, it appears that the main reason why
they are particularly tough is that we are quite ``lucky'' with one of
the models: it has simultaneously a very small bias, a very small size
and a large margin parameter, while other models with very similar
appearance are much worse.
When looking for a more general strong margin adaptivity result, we
then must keep in mind that this is a hopeless task in such situations.

Let us finally mention a related result in a close but slightly
different framework.
In the classification framework, under a \textit{global} margin
condition with $\varphi(x) \propto x^{2 \kappa}$ with $\kappa\geq1$,
Theorem 3 in~\cite{Lec2007b} shows that for any $M_n\geq2 $, a family
$( u_m )_{m \in\mathcal{M}_n}$ of $M_n$ classifiers exists for which,
for any
selection procedure $\widehat{m}$, some distribution $P$ exists such that
\[
\mathbb{E}[ P(f_{\widehat{m}} - f^{\star}) ] \geq\inf_{m \in
\mathcal{M}_n} \{ P(f_m - f^{\star} ) \} + C \biggl( \frac{\ln(M_n)}{n}
\biggr)^{\kappa/(2 \kappa- 1)} ,
\]
where $f_m = \gamma(u_m;\cdot)$ for some loss function $\gamma$.
When $\widehat{m}$ is (penalized) empirical minimization, the
remainder term is
shown to be as large as $C \sqrt{\ln(M_n)/n}$ when the margin condition
holds with $\kappa>1$.

This result and Theorem \ref{th.non-nested.new} focus on different
problems. In \cite{Lec2007b}, the margin condition is only assumed to
hold \textit{globally}, and the focus is on the dependence of the
remainder term on the cardinality $M_n$ of $\mathcal{M}_n $.
Therefore, the
counterexample given in \cite{Lec2007b} implies nothing about local
margin conditions for $(f_m)_{m \in\mathcal{M}_n} $.
Note that using these arguments, we could probably generalize Theorem
\ref{th.non-nested.new} to a family of $M_n \geq2$ functions and obtain
a lower bound depending on $M_n$ as in~\cite{Lec2007b}.

\section{General collections of models} \label{sec.oracle.gal}
As proven in Section \ref{sec.non-nested}, we cannot hope to obtain
margin adaptivity without any assumption on either $P$ or the models.
The purpose of this section is to explain what can still be proven in
the general case, and why this is weaker than our Theorem
\ref{th.nested}.

\subsection{A general oracle inequality}
We start with a general result for penalties satisfying the lower bound
(\ref{eq.pen.minor.nested}).

\begin{theorem} \label{th.gal.oracle}
Let $( \mathcal{F}_m )_{m \in\mathcal{M}_n}$ be any countable family
of models, and
$( t_m )_{m \in\mathcal{M}_n}$ be any sequence of positive numbers.
Let $\widehat{m}$ be defined by (\ref{eq.pen}) and assume that some
$c \in
(0,1)$ exists such that
%
\begin{equation} \label{eq.pen.minor.gal}
\forall m \in\mathcal{M}_n\qquad (1-c) \operatorname{pen}(m) \geq(P -
P_n) ( \widehat{f}_m - f_m ) +
\frac{t_m}{n} \geq0
\end{equation}
on an event of probability at least $1 - \eta$.

Then, there exists an event of probability at least $1 - \eta- 2
\sum_{m \in\mathcal{M}_n} \mathrm{e}^{-t_m}$ on which the following holds:
for every
$\varepsilon\in(0,1)$,
\begin{equation} \label{eq.gal.oracle}
P ( \widehat{f}_{\widehat{m}} - f^{\star}) \leq\frac{1}{1 -
\varepsilon} \inf
_{m \in\mathcal{M}_n} \biggl\{ P ( f_{m} - f^{\star}) + \operatorname
{pen}(m) + v(m) + \frac {t_m}{3n} \biggr\} + V_n,
\end{equation}
where
\[
V_n :=\frac{1}{1 - \varepsilon} \sup_{m \in\mathcal{M}_n}
\{ v(m) - \varepsilon P( f_{m} - f^{\star}) - c \operatorname{pen}(m)\}
\]
and
\[
v(m) :=\sqrt{ \frac{2t_m}{n} \operatorname{var}_P (f_m -
f^{\star}) }.
\]
\end{theorem}

Theorem \ref{th.gal.oracle} is proved in Section
\ref{sec.proofs.oracle}.
Let us make a few comments.

First, without $V_n$, (\ref{eq.gal.oracle}) is the kind of oracle
inequality we are looking for, since the leading constant is close to 1
(provided $\varepsilon$ is small enough).
For the sake of simplicity, assume that a~margin condition
(\ref{eq.margin.local.gal}) holds for every model $m \in\mathcal
{M}_n$, with
$\varphi_m(x) = h_m x^2$. Then,
\[
v(m) \leq\sqrt{ \frac{2t_m P
( f_m - f^{\star})}{h_m n}} \leq\varepsilon P ( f_m - f^{\star}) +
\frac{t_m}{2 \varepsilon h_m n}
\]
for any $\varepsilon\in(0,1)$.
Hence, the first term of the right-hand side of (\ref{eq.gal.oracle})
is smaller than
\[
\frac{1 + \varepsilon}{1 - \varepsilon} \inf_{m \in\mathcal{M}_n}
\biggl\{ P( f_m - f^{\star}) + \operatorname{pen}(m) + \frac{t_m}{ 2
\varepsilon h_m n} + \frac{t_m}{3 n} \biggr\} ,
\]
which is the right-hand side of a margin-adaptive oracle inequality
like (\ref{eq.oracle.margin}) (at least when the penalty is itself of
the right order).
A similar result holds for a more general $\varphi_m$; see the proof of
Theorem \ref{th.nested}.

Once we have a penalty satisfying (\ref{eq.pen.minor.gal}) (for
instance, a local Rademacher penalty), the main difficulty for proving
a strong margin adaptivity result then lies in $V_n$.
It arises from the difference between the ideal penalty and the
right-hand side of the lower bound (\ref{eq.pen.minor.gal}), that is
$(P-P_n)(f_m)$. This random quantity is centered, and (up to a quantity
independent of $m$) has deviations of order $v(m)$, Bernstein's
inequality being unimprovable. Then, if $v(m)$ happens to be much
larger than $P(f_m - f^{\star}) + \operatorname{pen}(m)$, $m$ is
selected with a positive
probability, whatever the value of $P(\widehat{f}_m - f^{\star}) $.
In that case,
the expectation of $\widehat{f}_{\widehat{m}}$ is worse than the
oracle by at least
$v(m)$ (for any of these ``bad'' models). Hence, $V_n$ certainly is
unavoidable in (\ref{eq.gal.oracle}).

As shown by Theorem \ref{th.non-nested.new}, $V_n$ can be much larger
than the expectation of a strong margin-adaptive estimator.
Nevertheless, $V_n$ is not always the main term on the right-hand side
of (\ref{eq.gal.oracle}). Let us now describe a set of favorable
situations in which it is possible to prove that $V_n$ is small enough:

\begin{enumerate}
\item Models are nested, $t_m$ is non-decreasing (with respect to the
inclusion ordering on $\mathcal{F}_m$), and $\operatorname{pen}$
satisfies the additional
condition (\ref{eq.nested.min-pen}); see Section \ref{sec.nested}.
\item Models are nested, $t_m$ is non-decreasing and $v(m)$ is
decreasing (or at least not increasing too much) when $\mathcal{F}_m$
increases.
Indeed, let us fix $m, m^{\star}\in\mathcal{M}_n$ (think of
$m^{\star}$ as a minimizer of
the infimum on the right-hand side of (\ref{eq.gal.oracle})). When
models are nested, either $\mathcal{F}_{m^{\star}} \subset\mathcal
{F}_{m}$ so that $v(m)
\leq
\sup_{\mathcal{F}_{m^{\star}} \subset\mathcal{F}_{m^{\prime}}} \{
v(m^{\prime}) \}
$, or
$\mathcal{F}_m \subset\mathcal{F}_{m^{\star}}$ so that $\varphi
_{m^{\star}} \leq\varphi_{m}$ hence
$\varphi_{m}^{\star} \leq\varphi_{m^{\star}}^{\star} $. In the
second case,
\[
v(m) - \varepsilon P(f_m - f^{\star}) \leq
\varphi_m^{\star}\Biggl( \sqrt{\frac{2 t_m}{\varepsilon^2 n}} \Biggr)
\leq
\varphi_{m^{\star}}^{\star}\Biggl( \sqrt{\frac{2 t_m}{\varepsilon^2 n}} \Biggr)
\leq\varphi_{m^{\star}}^{\star}\Biggl( \sqrt{\frac{2 t_{m^{\star
}}}{\varepsilon
^2 n}} \Biggr)
\]
since $t_m \leq t_{m^{\star}}$ and $\varphi^{\star}_{m^{\star}}$ is
non-decreasing.
As a consequence, for any $m^{\star}\in\mathcal{M}_n $,
\[
V_n \leq\frac{1}{1 - \varepsilon} \max\Biggl\{ \sup_{\mathcal
{F}_{m^{\star}} \subset\mathcal{F}_{m^{\prime}}} \{ v(m^{\prime})
\} ; \varphi_{m^{\star}}^{\star}\Biggl( \sqrt{\frac{2 t_{m^{\star
}}}{\varepsilon^2 n}} \Biggr) \Biggr\},
\]
which is not too large provided that $v(m)$ never increases too much.
Notice that we can understand assumption (\ref{eq.nested.min-pen}) as
ensuring that the penalty compensates a possible increase of $v(m)$.
\item The oracle model prediction error does not decrease to zero
faster than $n^{-1/2}$ and $t_m \leq t$. Indeed, the straightforward
upper bound $v(m) \leq\sqrt{2 t_m / n}$ shows that $V_n \leq(1 -
\varepsilon)^{-1} \sqrt{2 t / n}$.
\item The margin condition does not depend on $m$ and $t_m \leq t$.
Indeed, when $\varphi_m \equiv\varphi$ (or $\inf_m \varphi_m \geq
\varphi$), we have
\[
V_n \leq\frac{1}{1 - \varepsilon} \sup_{m \in\mathcal{M}_n} \Biggl\{
\varphi _m^{\star}\Biggl( \sqrt{\frac{2 t_m}{\varepsilon^2 n}} \Biggr) \Biggr\} \leq
\frac
{1}{1 - \varepsilon} \varphi^{\star}\Biggl( \sqrt{\frac{2 t}{\varepsilon
^2 n}} \Biggr).
\]
\item The penalty satisfies $c \operatorname{pen}(m) \geq v(m)$ for
every $m \in\mathcal{M}_n$,
which can be ensured for instance by adding $c^{-1} v(m)$ (or an
estimate of it) to a penalty satisfying (\ref{eq.pen.minor.gal}). An example of this
method is the one proposed by Koltchinskii \cite{Kol2006} (Section
5.2), and in that case (\ref{eq.gal.oracle}) coincides with his Theorem
6.
\end{enumerate}

Points 3 and 4 above show that the challenging situations are the ones
where the margin condition indeed depends on the model, and fast rates
of estimation are attainable. We prove in Section~\ref{sec.ex} that
such situations can occur, enlightening how our Theorem \ref{th.nested}
is an improvement on existing results and their straightforward
consequences.

On the other hand, point 5 may seem contradictory with the negative
results of Section \ref{sec.non-nested}. The explanation is that using
$v(m)$ in the penalty means that $\widehat{m}$ is not only a function
of the
data, but also of the unknown distribution $P$. Then it cannot be
considered adaptive. A more surprising consequence of this remark
combined with Theorem \ref{th.non-nested.new} is that $v(m)$ cannot be
estimated accurately enough uniformly over the set of all distributions
$P$. Consider the proposal, in Section~5.1 of \cite{Kol2006}, to add
\[
C \sqrt{ \frac{t_m P_n(\widehat{f}_m)}{n} }
\]
to the penalty, which is sufficient to give a result like
(\ref{eq.nested}). The point is that such a penalty is generally much
too large (at least for small models), which often results in an upper
bound of order $n^{-1/2}$. In the examples we have in mind (as well as
in the counterexamples of Section~\ref{sec.non-nested}), the excess
risk of the oracle is much smaller, typically of order $n^{-\beta}$ for
some $\beta\in(1/2;1]$.

\subsection{The local margin conditions can be significantly
tighter than the global one} \label{sec.ex}
In this section, we show that there exist challenging situations in
which the margin condition holds for functions $\varphi_m$ strongly
depending on $m$.

\begin{proposition} \label{le.local.vs.global.margin}
Let $\kappa\in(1;+\infty)$ and assume that $\mathcal{X}$ is
infinite. Let
$\gamma$ be the 0--1 loss and $\mathcal{F}^{0-1}:=
\{ \gamma( u;\cdot) \mbox{ \textup{s.t.} }u\dvtx \mathcal{X}\mapsto\{ 0,1 \}
\mbox{ is measurable} \}$ be the associated loss function class.
Then there exist a probability distribution $P$ on $\mathcal{X}\times
\{ 0,1 \} $, a sequence $(f_j)_{j \in\mathbb{N}}$ of elements of
$\mathcal{F}^{0-1}$
and positive constants $(C_i)_{5 \leq i \leq7}$ (depending on
$\kappa$ only) such that:
\begin{enumerate}[(iii)]
\item[(i)] $\forall k \in\mathbb{N}$, $P(f_{2k+1} - f^{\star}) =
P(f_{2k} - f^{\star}
) =
b(k)$ and $2^{-k \kappa-2} \leq b(k) \leq2^{-k \kappa-1} $.
\item[(ii)] The global margin condition (\ref{eq.margin.global.gal}) is
satisfied over $\mathcal{F}=\mathcal{F}^{0-1}$ with $\varphi(x) =
C_5 x^{2\kappa} $, and
it is tight: $\forall k \in\mathbb{N},
\varphi( \sqrt{\operatorname{var}( f_{2k+1} - f^{\star})} ) \geq C_6
P(f_{2k+1} -
f^{\star}) $.
\item[(iii)] A tighter local margin condition
(\ref{eq.margin.local.gal}) holds over $\{ f_{2k} \mbox{ s.t. }k
\in \mathbb{N} \} \dvtx \forall k \in\mathbb{N}$,
$P(f_{2k}-f^{\star}) \leq\operatorname{var}_P
(f_{2k} - f^{\star}) $.
\item[(iv)] For every $m \in\mathbb{N}$, define $\mathcal{F}_m = \{
f_m \}$ and
consider the model selection problem among $( \mathcal{F}_m )_{0 \leq m
\leq
M_n}$ with $M_n \geq2 \ln_2(n) $. Then, the right-hand side of a
strong margin-adaptive oracle inequality of the form
(\ref{eq.oracle.gal}) is at most proportional to
\[
\inf_{0 \leq2 k \leq M_n } \biggl\{ P(f_{2k} - f^{\star}) + \frac{\ln
(n)}{n} \biggr\} \leq\frac{2 \ln(n)}{n} ,
\]
whereas the right-hand side of a global margin-adaptive oracle
inequality is larger than $C_7 n^{-\kappa/(2\kappa-1)} \gg(\ln(n))/n
$.
\end{enumerate}
\end{proposition}

Proposition \ref{le.local.vs.global.margin} is proved in Section
\ref{sec.pr.le.local.vs.global.margin}.
It gives an example of a model selection problem where \textit{strong
margin adaptivity implies a faster rate of convergence than adaptivity
to the global margin condition}.
Note that the same argument works with many other model selection
problems, such as selecting among $( \{ f_{2k+1} \mbox{ s.t. }0 \leq
k \leq m \} )_{m \in\{ 1, \ldots, (\ln(n))^2 \}} $.

\section{Discussion} \label{sec.discussion}

\subsection{Other penalization procedures}
We have focused in Section \ref{sec.setting.loc-complex} on penalties
defined in terms of local Rademacher complexities in order to prove
that strong margin adaptivity is attainable for some data-driven
penalties. An~interesting question is whether such a result can be
extended to penalties that can be computed faster.

For instance, it is natural to think of estimating $\operatorname
{pen}_{\mathrm{id}}(m)$ itself
by resampling, instead of the local complexity
$\overline{\delta}_n(\mathcal{F}_m ; t)$. Such penalties, with
several kinds of
resampling schemes, have been proposed in \cite{Arl2008a} and
\cite{Arl2009RP} and called ``resampling penalties'' (RP), generalizing
the bootstrap penalty suggested by Efron \cite{Efr1983}.
Resampling penalties can be computed faster than local Rademacher
complexities, because they are not defined as fixed points of the
resampling estimate of a function. In particular, the $V$-fold
penalties defined in \cite{Arl2008a} have the same computational cost
as $V$-fold cross-validation.

In addition, RP are easy to calibrate, since they depend on a single
tuning parameter -- the multiplicative factor in front of it -- which
can, for instance, be estimated from the data by using the ``slope
heuristics'' (see \cite{ArlMas2009pente}).
On the contrary, local Rademacher complexities depend on two more
constants, whose theoretical values are certainly too large for
practical application.

Extending Corollary \ref{cor.nested.locRad} to RP would require to
prove that RP satisfy both assumptions (\ref{eq.pen.minor.nested}) and
(\ref{eq.nested.min-pen}). On the one hand, (\ref{eq.pen.minor.nested})
means essentially that the penalty is larger than the expectation of
the ideal penalty with large probability. Hence, one can conjecture
that (\ref{eq.pen.minor.nested}) holds for RP; a~partial proof of
(\ref{eq.pen.minor.nested}) for RP in our general setting can be found
in Chapter 7 of \cite{Arl2007phd}, together with an agenda for a
complete proof, which seems to be a difficult theoretical problem.
On the other hand, (\ref{eq.nested.min-pen}) seems less likely to hold
for RP, and we may have to modify RP so that (\ref{eq.nested.min-pen})
can be satisfied in general.

Proving such results would be quite interesting, since it would provide
a strong margin-adaptive penalization procedure with a reasonably small
computational cost.

\subsection{Should we make collections of models nested?}

A natural question coming from our results is whether one should make
any collection of models nested before performing model selection in
order to improve performance. Let us consider the counterexample of
Theorem \ref{th.non-nested.new} and look at what would happen if we
make the models nested.

Assume that $P = P_1$ is the distribution defined in the proof of
Theorem \ref{th.non-nested.new}. On the one hand, comparing $\{ f_0 \}$
and $\{ f_0 , f_1 \} $, the model selection problem would be easy
because the margin parameter $h_m$ is the same in both models, making
the remainder term of order $n^{-1/2}$ (the remainder term $(n
h_m)^{-1}$ can be replaced by $n^{-1/2}$ when $h_m \leq n^{-1/2}$
because of the upper bound $\operatorname{var}_P (f_m - f^{\star})
\leq1/4$). And margin
adaptivity is not challenging when the margin condition is merely not
satisfied.
On the other hand, when $P = P_1$, comparing $\{ f_1 \}$ and $\{ f_0 ,
f_1 \}$ is more challenging because $f_1$ is really better than
$f_0 $. Here, contrary to the non-nested case, the large increase of
the term $\operatorname{var}_P (f_m - f^{\star})$ induces a similar
increase in the
$L^2(P_1)$ diameter of the class. Hence, local Rademacher complexities
can detect it, as shown by Theorem \ref{th.nested}.

To conclude, improving significantly the prediction performance of the
final estimator by making the models nested requires some prior
knowledge, such as a natural ordering between the (non-nested) models.
Otherwise, Theorem \ref{th.non-nested.new} shows that choosing how to
make the models nested, either from data or randomly, is not successful
with probability at least $C_3>0 $, whatever the sample size.

\section{Proofs} \label{sec.proofs}

\subsection{Oracle inequalities} \label{sec.proofs.oracle}

We give the proofs in a logical order, that is, first Theorem
\ref{th.gal.oracle}, then Theorem \ref{th.nested} (which is a corollary
of it), and finally Corollary \ref{cor.nested.locRad}.

\begin{pf*}{Proof of Theorem \ref{th.gal.oracle}}
First, by definition of $\widehat{m}$, for every $m \in\mathcal
{M}_n$ we have
\begin{eqnarray*}
&&P_n ( \widehat{f}_{\widehat{m}} ) + \operatorname{pen}(\widehat{m})\\
&&\quad\leq P_n ( \widehat{f}_m ) + \operatorname{pen}
(m),
\end{eqnarray*}
which can be rewritten as
\begin{eqnarray*}
&& P ( \widehat{f}_{\widehat{m}} - f^{\star}) + (P_n - P)( \widehat
{f}_{\widehat{m}} - f_{\widehat{m}} ) + (P_n
- P)( f_{\widehat{m}} - f^{\star}) + \operatorname{pen}(\widehat
{m}) \\
&&\quad\leq
P ( f_{m} - f^{\star}) + P_n( \widehat{f}_{m} - f_{m} ) + (P_n -
P)( f_{m} - f^{\star}) + \operatorname{pen}(m) \\
&&\quad\leq P ( f_{m} - f^{\star}) + (P_n - P)( f_{m} - f^{\star}) +
\operatorname{pen}(m).
\end{eqnarray*}
In the event that (\ref{eq.pen.minor.gal}) holds, we then have
%
\begin{eqnarray} \label{eq.gal.oracle.tmp1}
&& P( \widehat{f}_{\widehat{m}} - f^{\star}) + (P_n -
P)( f_{\widehat{m}} - f^{\star}) +
c \operatorname{pen}(\widehat{m}) + \frac{t_{\widehat{m}}}{n} \nonumber\\ [-8pt]\\
[-8pt]
&&\quad\leq\inf_{m \in\mathcal{M}_n}
\{ P ( f_{m} - f^{\star}) + (P_n - P)( f_{m} - f^{\star}) +
\operatorname{pen} (m) \}.\nonumber
\end{eqnarray}

By Bernstein's inequality (see, e.g., Proposition 2.9 in
\cite{Mas2003StFlour}), for every $m \in\mathcal{M}_n$, there is an
event of
probability $1 - 2 \mathrm{e}^{-t_m}$ on which
\[
| (P_n - P)( f_{m} - f^{\star}) | \leq v(m) + \frac{t_m}{3n} .
\]
On the intersection of these events with the one in which
(\ref{eq.pen.minor.gal}) holds, we derive from
(\ref{eq.gal.oracle.tmp1}) that
\[
P ( \widehat{f}_{\widehat{m}} - f^{\star}) - v(\widehat{m}) + c
\operatorname{pen}(\widehat{m}) \leq\inf_{m \in\mathcal{M}_n}
\biggl\{ P ( f_{m} - f^{\star}) + \operatorname{pen}(m) + v(m) + \frac
{t_m}{3n} \biggr\}.
\]
For any $\varepsilon> 0$, the left-hand side is larger than
\begin{eqnarray*}
&& (1 - \varepsilon) P ( \widehat{f}_{\widehat{m}} -
f^{\star}) +
\varepsilon P( f_{\widehat{m}} - f^{\star}) + c \operatorname
{pen}(\widehat{m}) - v(\widehat{m}) \\
&&\quad\geq
(1 - \varepsilon) P ( \widehat{f}_{\widehat{m}} - f^{\star}) - \sup
_{m \in\mathcal{M}_n} \{ v(m) - \varepsilon P( f_{m} - f^{\star}) -
c \operatorname{pen}(m) \}.
\end{eqnarray*}
The result follows.
\end{pf*}

\begin{pf*}{Proof of Theorem \ref{th.nested}}
We consider the event in which (\ref{eq.gal.oracle}) holds. By Theorem
\ref{th.gal.oracle}, we know that it has probability at least $1 -
\eta- 2 \sum_{m \in\mathcal{M}_n} \mathrm{e}^{-t_m}$.
We first bound the first term on the right-hand side of
(\ref{eq.gal.oracle}). From (\ref{eq.margin.local.gal}), we have
\[
\forall m \in\mathcal{M}_n \qquad v(m) \leq\sqrt{ \frac
{2t_m}{n}} \varphi_m^{-1}
\bigl( P ( f_m - f^{\star}) \bigr).
\]
Then, using that $xy \leq\varphi_m(x) + \varphi_m^{\star}(y)$ for
every $x,y \geq0$,
\[
\forall m \in\mathcal{M}_n \qquad v(m) \leq\varphi_m^{\star} \Biggl(
\sqrt{ \frac
{2t_m}{ \varepsilon^2 n}} \Biggr) + \varphi_m\bigl( \varepsilon\varphi_m^{-1}
\bigl( P ( f_m - f^{\star}) \bigr) \bigr).
\]
Since $\varphi_m$ is convex with $\varphi_m(0)=0$, we have
$\varphi_m(\lambda x) \leq\lambda\varphi_m(x)$ for every $\lambda
\in
(0,1)$ and $x \geq0$. Then, using also that $\operatorname{var}_P
(f_m - f^{\star}) \leq
1$,
%
%
\begin{equation} \label{eq.maj.v-margin-loc}
\forall m \in\mathcal{M}_n \qquad v(m) \leq\min\Biggl\{\sqrt{
\frac{2 t_m}{n} } ,\varphi_m^{\star} \Biggl( \sqrt{ \frac{2t_m}{
\varepsilon^2 n}} \Biggr) + \varepsilon P ( f_m - f^{\star}) \Biggr\},
\end{equation}
and the right-hand side of (\ref{eq.gal.oracle}) is smaller than
%
\begin{equation} \label{eq.nested.tmp1}
\hspace*{-10pt}\frac{1}{1 - \varepsilon} \inf_{m \in\mathcal{M}_n} \Biggl\{ (1 +
\varepsilon) P ( f_m - f^{\star}) + \operatorname{pen}(m) + \min
\Biggl\{\varphi_m^{\star} \Biggl( \sqrt{ \frac {2 t_m}{\varepsilon^2 n}
} \Biggr) ,\sqrt{\frac{2 t_m}{n}} \Biggr\} + \frac{t_m}{3n} \Biggr\} + V_n.
\end{equation}
It now remains to upperbound $V_n$.

Let $m,m^{\prime}\in\mathcal{M}_n$. Since models $\mathcal{F}_m$
are nested, two cases can
occur:
\begin{enumerate}
\item$\mathcal{F}_m \subset\mathcal{F}_{m^{\prime}}$, which
implies $t_m \leq t_{m^{\prime}}$ and
$\varphi_m \geq\varphi_{m^{\prime}}$, hence $\varphi_m^{\star}
\leq
\varphi_{m^{\prime}}^{\star}$. Using, in addition,
(\ref{eq.maj.v-margin-loc}) and that $\varphi_{m^{\prime}}^{\star}$ is
non-decreasing , we have
\begin{eqnarray*}
v(m)\leq\min\Biggl\{\sqrt{ \frac{2 t_{m^{\prime}}}{n} } ,\varphi
_{m^{\prime} }^{\star} \Biggl( \sqrt{ \frac{2 t_{m^{\prime}}}{
\varepsilon^2 n}} \Biggr) \Biggr\} +
\varepsilon P( f_m - f^{\star}).
\end{eqnarray*}

\item$\mathcal{F}_{m^{\prime}} \subset\mathcal{F}_{m}$. Using
(\ref{eq.nested.min-pen}) and
(\ref{eq.maj.v-margin-loc}),
\begin{eqnarray*}
v(m) &\leq& C_1 v(m^{\prime}) + C_2 P( f_{m^{\prime}} - f^{\star}) +
c \operatorname{pen}(m) \\
&\leq& C_1 \min\Biggl\{\sqrt{ \frac{2 t_{m^{\prime}}}{n} } ,\varphi
_{m^{\prime}}^{\star} \Biggl( \sqrt{ \frac{2 t_{m^{\prime}}}{
\varepsilon^2 n}} \Biggr)\Biggr\} + (C_2 + C_1 \varepsilon) P(
f_{m^{\prime}} - f^{\star}) + c \operatorname{pen}(m).
\end{eqnarray*}
\end{enumerate}
Therefore,
\[
V_n \leq\frac{1}{1-\varepsilon} \inf_{m^{\prime}\in\mathcal
{M}_n} \Biggl\{ \max \{ 1, C_1 \} \min\Biggl\{\sqrt{ \frac{2
t_{m^{\prime}}}{n} } ,\varphi _{m^{\prime}}^{\star} \Biggl( \sqrt{ \frac
{2 t_{m^{\prime}}}{ \varepsilon^2 n}} \Biggr) \Biggr\} + (C_2 + C_1
\varepsilon) P( f_{m^{\prime}} - f^{\star}) \Biggr\}
\]
and the result follows.
\end{pf*}

\begin{pf*}{Proof of Corollary \ref{cor.nested.locRad}}
From \cite{Kol2006} (Theorem 1 and (9.2) in the proof of its Lemma 2),
we know that there exist numerical constants $\overline{K}>0$ and
$q>1$ such that (\ref{eq.pen.minor.nested}) holds with $t_m = t$, $c =
5/7$ and $\eta= (L+1) \ln_q ( \frac{n}{t} ) \operatorname
{Card}(\mathcal{M}_n) \mathrm{e}^{-t}$.

In addition, Lemma \ref{le.min.loc-Rad.diam} below shows that
(\ref{eq.nested.min-pen}) holds with $C_1 = \sqrt{2}$ and $C_2 =
2/(\overline{K} q)$.

The result follows from Theorem \ref{th.nested} with $t_m = t$.
\end{pf*}
\setcounter{lemma}{2}
\begin{lemma} \label{le.min.loc-Rad.diam}
Let $\mathcal{F}_{m^{\prime}}\subset\mathcal{F}_m $ and $\overline
{\delta}_n$ be defined
by (\ref{eq.loc-comp.ideal}). Then,
%
\begin{equation} \label{eq.min.loc-Rad.diam}
v(m) \leq2 \overline{\delta}_n(\mathcal{F}_m ; t) + \sqrt{2}
v(m^{\prime}) +
\frac{2 P( f_{m^{\prime}} - f^{\star})}{q\overline{K}} .
\end{equation}
\end{lemma}

\begin{pf}
Since $\mathcal{F}_{m^{\prime}}\subset\mathcal{F}_m $,
$f_{m^{\prime}} \in\mathcal{F}_m$ (as
well as $f_m$), so that
\begin{eqnarray}\label{le.min.loc-Rad.diam.tmp1}
D_P \bigl( \mathcal{F}_m ; P( f_{m^{\prime}} - f_m)\bigr) &\geq&\sqrt{ P( f_m -
f_{m^{\prime}} )^2}
\geq\sqrt{\operatorname{var}_P ( f_m - f_{m^{\prime}} )} \nonumber\\
[-8pt]\\ [-8pt]
&\geq&\sqrt{\frac{\operatorname{var}_P ( f_m - f^{\star})}{2}} -
\sqrt{\operatorname{var}_P
( f_{m^{\prime}} - f^{\star})}.\nonumber
\end{eqnarray}
For the last inequality, we used that $\operatorname{var}(X) \leq2
\operatorname{var}(X+Y) + 2
\operatorname{var}(Y)$ for any random variables $X,Y$, and the
inequality $\sqrt{x+y}
\leq\sqrt{x} + \sqrt{y}$ for every $x, y \geq0$.

First, assume that the lower bound in (\ref{le.min.loc-Rad.diam.tmp1})
is non-positive. This implies
\[
v(m) = \sqrt{\frac{2 t}{n} \operatorname{var}_P ( f_m - f^{\star})}
\leq\sqrt
{2} v( m^{\prime} ),
\]
so that (\ref{eq.min.loc-Rad.diam}) holds.

Otherwise, the assumptions of Lemma \ref{le.min.loc-Rad.diam.tech}
below hold with
\[
D_0 = \sqrt{\frac{\operatorname{var}_P ( f_m - f^{\star})}{2}} -
\sqrt{\operatorname{var}_P
( f_{m^{\prime}} - f^{\star})} > 0
\]
and
\[
 \sigma_0
= P( f_{m^{\prime}} - f_m).
\]
We deduce from (\ref{eq.min.loc-Rad.diam.tech}) that
\[
\frac{v(m)}{2} - \frac{v(m^{\prime})}{\sqrt{2}} \leq\overline
{\delta}_n(\mathcal{F}_m ; t) + \frac{P( f_{m^{\prime}} -
f_m)}{q\overline
{K}} \leq\overline{\delta}_n(\mathcal{F}_m ; t) + \frac{P(
f_{m^{\prime}}
- f^{\star})}{q\overline{K}} ,
\]
and (\ref{eq.min.loc-Rad.diam}) also holds.
\end{pf}

\begin{lemma} \label{le.min.loc-Rad.diam.tech}
Let $\overline{\delta}_n(\mathcal{F}_m ; t)$ be defined by
(\ref{eq.loc-comp.ideal}). Assume that there is some $D_0,\sigma_0>0$
such that $D_P(\mathcal{F}_m ; \sigma_0) \geq D_0$. Then, we have the
following
lower bound:
%
\begin{equation} \label{eq.min.loc-Rad.diam.tech}
\max\biggl\{ \overline{\delta}_n(\mathcal{F}_m ; t) ; \frac{\sigma_0}{q
\overline{K}} \biggr\} \geq D_0 \sqrt{ \frac{t}{n}} .
\end{equation}
\end{lemma}
\begin{pf}
First, (\ref{eq.min.loc-Rad.diam.tech}) clearly holds when
$\frac{\sigma_0}{q \overline{K}} \geq D_0 \sqrt{t/n}$.
Otherwise, let $\sigma_1 = \max\{ q \overline{K} , 1 \}\break D_0
\sqrt{t/n} > \sigma_0$.
From the definition of $U_n $, we have
\[
\frac{U_n (\mathcal{F}_m ; \sigma_1 ; t)}{\sigma_1} \geq\frac
{\overline
{K} D_P(\mathcal{F}_m ; \sigma_1)}{\sigma_1} \sqrt{ \frac{t}{n} }
\geq
\frac{\overline{K} D_0}{q \overline{K} D_0 \sqrt{t/n}} \sqrt{
\frac{t}{n} } = \frac{1}{q} > \frac{1}{2q} .
\]
Then, according to the definition (\ref{eq.loc-comp.ideal}) of
$\overline{\delta}_n(\mathcal{F}_m ; t)$, $\overline{\delta
}_n(\mathcal{F}_m ; t)
\geq
\sigma_1 \geq D_0 \sqrt{t/n}$ and the result follows.
\end{pf}

\subsection{\texorpdfstring{Lower bounds (proof of Theorem \protect\ref{th.non-nested.new})}%
{Lower bounds (proof of Theorem 2}}
\label{sec.proofs.lower}
For every $m\in\{ 0,1 \} $, let $f_m \dvtx (x,y) \mapsto\mathbh{1}_{y
\neq
m} $; $f_m \in\mathcal{F}^{0-1}$, since $f_m = \gamma( u_m ; \cdot)$,
where for every $x \in\mathcal{X}$, $u_m(x)=m $.
Let $\alpha=(2n)^{-1}$ and $h=(2n)^{-1/2} $. Let $a \neq b$ be any two
elements of $\mathcal{X}$. We define a probability distribution $P_1$
on $\mathcal{X}
\times\{ 0,1 \}$ as follows: if $(X,Y)\sim P_1 $, then
$\mathbb{P}(X=a)=\alpha$, $\mathbb{P}(X=b)=1-\alpha$, $\mathbb{P}(
Y=1 \vert
X=a )=0$ and $\mathbb{P}( Y=1 \vert X=b )=\frac{1}{2}+h $.
We also define $P_0$ as the distribution of $(X,1-Y),$ where $(X,Y)\sim
P_1 $. In the following, for any distribution $Q$ on $\mathcal
{X}\times
\{ 0,1 \} $, we use the notation $\mathbb{P}_Q$ as a shortcut for
$\mathbb{P}_{(X_i,Y_i)_{1 \leq i \leq n} \sim Q^{\otimes n}}$.

First, under distribution $P_1 $, the Bayes predictor is
$s=\mathbh{1}_{b} $,
\[
P_1 ( f_0 - f^{\star}) = 2(1-\alpha)h,\qquad
P_1 ( f_1 - f^{\star}) = \alpha
\quad\mbox{and}\quad \operatorname{var}_{P_1} ( f_1 - f^{\star}) = \alpha-
\alpha^2.
\]
Hence,
\begin{eqnarray*}
&& \min_{m \in\{ 0,1 \}} \biggl\{ P_1 ( f_m - f^{\star}) + \overline{v}
(m) + \frac{\ln(n)}{n h_m} \biggr\} \\
&&\quad\leq P_1(f_1-f^{\star})+\overline{v}(1)+ \frac{\ln(n)}{n h_1} \leq
\alpha+ \sqrt{
\frac{2 \alpha\ln(n)}{n}} + \frac{\ln(n)}{n} \leq\frac{2+3\ln(n)}{2n}.
\end{eqnarray*}
Therefore, if $\mathbb{P}_{P_1} ( \widehat{m}=0 ) \geq C_3 $, then
(\ref{eq.non-nested.minimax.prob}) holds when $P=P_1 $, with
$C_4=1/3 $.
Similarly, $\mathbb{P}_{P_0} ( \widehat{m}=1 ) \geq C_3$ implies
(\ref{eq.non-nested.minimax.prob}) with $P=P_0$ and $C_4=1/3 $.
So, in order to prove (\ref{eq.non-nested.minimax.prob}), we only need
to prove that
%
\begin{equation} \label{eq.non-nested.minimax.suff} \max_{j \in\{
0,1 \}} \{ \mathbb{P}_{P_j} ( \widehat{m}=1-j ) \} \geq C_3 > 0.
\end{equation}

The proof of (\ref{eq.non-nested.minimax.suff}) relies on three main
facts. First,
%
\begin{equation} \label{eq.non-nested.minimax.proba-X}
\forall j \in
\{ 0,1 \} \qquad \mathbb{P}_{P_j} ( \forall i , X_i = b ) =
( 1 - \alpha)^n = \biggl( 1 - \frac{1}{2n} \biggr)^{n} \geq\frac
{1}{2} .
\end{equation}

Second, for every $j \in\{ 0,1 \} $, under $P_j $, conditionally to
$\{ \forall i , X_i = b \} $, $\operatorname{Card}\{ i \mbox{ s.t.
}Y_i = 1 \}$ is
a binomial random variable with parameters $(n,p_j) $, where
\[
p_j = \mathbb{P}_{(X,Y) \sim P_j} ( Y = 1 ) = \frac{1}{2} +
(-1)^{j+1} h.
\]
So, Lemma \ref{le.binom.minor} shows that for every $j \in\{ 0,1 \}$
and every $k \in\mathbb{N}\cap[ \frac{n}{2} - \sqrt{n} , \frac
{n}{2} + \sqrt{n} ] $,
%
%
\begin{equation} \label{eq.non-nested.minimax.binom-minor-p}
\mathbb{P}_{P_j} ( \operatorname{Card}\{ i \mbox{ s.t. }Y_i = 1 \}
= k \vert\forall i ,
X_i = b ) \geq\frac{C}{\sqrt{n}} > 0,
\end{equation}
where $C$
is an absolute constant.

Third, let us define, for every $k \in\{ 0, \ldots, n \} $,
\[
\pi_k :=\mathbb{P}_{P_U} \bigl( \widehat{m}( (X_i,Y_i)_{1 \leq i \leq n}
) = 1
\vert\operatorname{Card}\{ i \mbox{ s.t. }Y_i = 1 \} = k \mbox{
and } \forall i ,
X_i=b \bigr),
\]
where $P_U$ is the uniform distribution on $\{ a,b \} \times
\{ 0,1 \}$. A crucial remark is that $P_U$ can be replaced by either
$P_0$ or $P_1$ in the definition of $\pi_k $, since the conditioning
event determines $(X_i,Y_i)_{1\leq i \leq n}$ up to the ordering of the
observations; in the definition of $\pi_k $, the probability only
refers to the ordering of the $(X_i,Y_i) $, and any product measure on
$\mathcal{X}\times\{ 0,1 \}$ assigns equal probabilities to the $n!$
permutations of the $n$ observations.
Note also that the definition of $\pi_k$ stays valid when $\widehat
{m}$ is a
randomized selection rule, which proves the generalization of Theorem~\ref{th.non-nested.new} pointed out in Remark
\ref{rk.non-nested.minimax.random}.
For any given selection rule $\widehat{m}$,
\[
\operatorname{Card}\biggl\{ k \in\mathbb{N}\cap\biggl[ \frac{n}{2} - \sqrt{n}
, \frac{n}{2} + \sqrt{n} \biggr] \mbox{ s.t. }\pi_k > \frac{1}{2} \biggr\}
\]
is either larger or smaller than $\sqrt{n} $.
If it is larger, (\ref{eq.non-nested.minimax.proba-X}),
(\ref{eq.non-nested.minimax.binom-minor-p}) and the definition of the
$\pi_k$ (with $P_0$ instead of $P_U$) show that
\[
\mathbb{P}_{P_0} \bigl( \widehat{m}( (X_i,Y_i)_{1 \leq i \leq n} ) = 1 \bigr)
\geq\sqrt{n} \times\frac{C}{\sqrt{n}} \times\frac{1}{2} = \frac
{C}{2} = C_3 > 0,
\]
so that (\ref{eq.non-nested.minimax.suff}) is satisfied.
Otherwise, choosing $P_1$ instead of $P_0$ shows that
(\ref{eq.non-nested.minimax.suff}) holds true.
This proves (\ref{eq.non-nested.minimax.prob}), which clearly implies
(\ref{eq.non-nested.minimax.esp}), since $P(f_{\widehat{m}}-f^{\star
})\geq0$ a.s.

A key tool in the proof of Theorem \ref{th.non-nested.new} is the
following uniform lower bound on the density of the binomial
distribution w.r.t. the counting measure on $\mathbb{N}$.
\begin{lemma} \label{le.binom.minor}
For every $n \in\mathbb{N}$ and $p \in[0,1]$, let $\mathcal
{B}(n,p)$ denote
the binomial distribution with parameters $(n,p)$.
For every $a,b > 0$ and $c \in(0,1/2)$, a positive constant $C(a,b,c)$
exists such that for any positive integer $n $,
%
\begin{equation} \label{eq.binom.minor}
\mathop{\inf_{k\in\mathbb{N}, | k - n/2 | \leq\min\{ a
n^{1/2},n/2 \}}}_{| p - 1/2 | \leq\min\{ b
n^{-1/2}, c \} } \bigl\{ \sqrt{n} \mathbb{P}_{Z \sim\mathcal{B}(n,p)}( Z
= k ) \bigr\} \geq C (a,b,c) > 0.
\end{equation}
\end{lemma}

\begin{pf}
Let $n,k,p$ satisfy the above conditions, $Z \sim\mathcal{B}(n,p)$,
and define
\[
\eta:=\frac{2k}{n} - 1, \qquad \delta:=p - \frac{1}{2} .
\]
The assumption on $k$ and $p$ becomes $| \eta |
\leq\min\{ a n^{-1/2},1/2 \}$ and $| \delta | \leq\min\{ b
n^{-1/2},c \}$.
In addition,
\[
\mathbb{P}( Z = k ) = p^k (1-p)^{n-k} {n\choose k} = \biggl( \frac{1}{2} +
\delta\biggr)^k \biggl( \frac{1}{2} + \delta\biggr)^{n-k} \frac{n!
}{k! (n-k)!} .
\]

We now use Stirling's formula:
\[
\ln(n!) = n \ln(n) - n + \frac{1}{2} \ln( 2 \uppi n ) + \varepsilon
_n
\]
for some sequence $\varepsilon_n \rightarrow0$ when $n \rightarrow+
\infty$ (one has $(12 n+1)^{-1} \leq\varepsilon_n \leq(12 n)^{-1}$).
Then,
\begin{eqnarray*}
\ln\mathbb{P}( Z = k ) &=& k \ln\biggl( \frac{1}{2} + \delta\biggr) +
(n-k) \ln\biggl( \frac{1}{2} - \delta\biggr) + \ln\frac{n!}{k! (n-k)!}
\\
&=& \frac{n}{2} \biggl[ (1 - \eta) \ln\biggl( \frac{1 - 2 \delta}{1 - \eta} \biggr)
+ (1 + \eta) \ln\biggl( \frac{1 + 2 \delta}{1 + \eta} \biggr) \biggr] \\
&&{} - \frac{1}{2} \ln(n) + \frac{1}{2} \ln\biggl( \frac{2}{\uppi} \biggr) -
\frac{1}{2} \ln( 1 - \eta^2 ) + \varepsilon_n - \varepsilon
_k -
\varepsilon_{n-k} .
\end{eqnarray*}
Define $h \dvtx (-1 , + \infty) \mapsto\mathbb{R}$ by $h(x) :=x^{-1}
\ln(1+x) - 1$, so that
\[
\forall x > -1\qquad \ln(1 + x) = x \bigl( 1 + h(x) \bigr).
\]
Recall that $ | h(x) | \leq2| x |$ as soon as $x \geq- 1/2$, by
the Taylor--Lagrange formula. In particular, $\lim_{x \rightarrow0}
h(x) = 0$.
We then have
\begin{eqnarray*}
\ln\mathbb{P}( Z = k ) &=& \frac{n}{2} [ 4 \delta\eta- 2 \eta^2 - 2\delta(1 - \eta ) h(-2
\delta) + \eta(1-\eta) h(-\eta)\\
&&\hphantom{\frac{n}{2} [}{} + 2 \delta(1 + \eta) h(2 \delta)
- \eta(1 + \eta) h(\eta) ] \\
&&{} - \frac{1}{2} \ln(n) + \frac{1}{2} \ln\biggl( \frac{2}{\uppi} \biggr) +
\frac{\eta^2}{2} h(- \eta^2) + \varepsilon_n - \varepsilon_k -
\varepsilon_{n-k} .
\end{eqnarray*}
Assuming that $n \geq n_0$ such that $\max\{ a,b \} n^{-1/2} \leq1/2$,
it follows that
\[
\ln\mathbb{P}( Z = k ) = - \frac{1}{2} \ln(n) + R(k,n,p)
\]
with
\[
R(k,n,p) \geq L ( 1 + a^2 + ab + b^2 )
\]
for some numerical constant $L>0$, and this lower bound is uniform over
$n \geq n_0$ and $k, p$ such that the conditions of the infimum in
(\ref{eq.binom.minor}) are satisfied.
On the other hand,
\[
\inf_{n \leq n_0, 1 \leq k \leq n} \bigl\{ \mathbb{P}_{Z \sim\mathcal
{B}(n,p)}( Z = k ) \bigr\} \geq K(p) > 0
\]
as soon as $p \in(0,1)$. Since $\mathbb{P}_{Z \sim\mathcal
{B}(n,p)}( Z = k )$, seen as a function of $p$, is increasing on
$(0,k/n)$ and decreasing on $(k/n,1)$, $K(p) $ is uniformly larger than
$\min\{ K(1/2 - c),\break K(1/2 + c) \}$. The result follows.
\end{pf}

\subsection{\texorpdfstring{Proof of Proposition \protect\ref{le.local.vs.global.margin}}%
{Proof of Proposition 1}}
\label{sec.pr.le.local.vs.global.margin}
Let $( x_j )_{j \in\mathbb{N}}$ be any infinite sequence of distinct
elements of $\mathcal{X}$ and $\lambda> 0$ to be chosen later. We
define $P$ as
follows, by denoting $(X,Y)$ a pair of random variables with joint
distribution $P $.
For every $k \in\mathbb{N}$, $\mathbb{P}( X = x_{2k} ) = p_k q_k$
and $\mathbb{P}( X =
x_{2k +1}) = p_k (1 - q_k) $, where $p_k = 2^{-k-1}$ and $q_k \in
[0,1]$ is to be chosen later; note that $\sum_{k \in\mathbb{N}} p_k
= 1 $.
For every $k \in\mathbb{N}$, $\mathbb{P}( Y = 1 \vert X =x_{2k} ) = 0$
and $\mathbb{P}( Y = 1 \vert X =x_{2k+1} ) = (1 + \delta_k)/2$
where $\delta_k = 2^{-k \lambda}$.
As a consequence, the Bayes predictor is $s:=
\mathbh{1}_{\{ x_{2k + 1} \mbox{ s.t. }k \in\mathbb{N} \}} $.
Let us define for every $j \in\mathbb{N}$,
\[
u_j (x) :=
\cases{
s(x), &\quad if  $x \neq x_j$, \cr
1 - s(x), &\quad if  $x = x_j$
}\quad
\mbox{and}\quad
f_j = \gamma( u_j;\cdot),
\]
where $\gamma$ is the 0--1 loss.
Then, for any $k \in\mathbb{N}$,
\begin{eqnarray} \label{eq.pr.le.local.vs.global.margin.1}
P ( f_{2k+1} - f^{\star}) &=& \delta_k p_k (1 - q_k),\qquad P ( f_{2k} -
f^{\star}) =p_k q_k, \\\label{eq.pr.le.local.vs.global.margin.2}
\operatorname{var}_P ( f_{2k+1} - f^{\star}) &=& p_k ( 1- q_k) - \bigl(\delta_k p_k (1 - q_k)
\bigr)^2,
\\\label{eq.pr.le.local.vs.global.margin.3}
\operatorname{var}_P (f_{2k} - f^{\star}) &=&p_k q_k - ( p_k q_k )^2.
\end{eqnarray}
We can now prove the four statements of Proposition
\ref{le.local.vs.global.margin}.
\begin{enumerate}[(iii)]
\item[(i)] By (\ref{eq.pr.le.local.vs.global.margin.1}), choosing $q_k
= \delta_k/(1+\delta_k)$ and $\lambda=\kappa-1>0$ implies (i) with
$b(k) = p_k q_k $.
\item[(ii)] For every $t \in(0,1) $,
%
\begin{equation} \label{eq.local.vs.global.margin.tmp1}
\mathbb{P}\bigl( | 2 \eta(X) - 1 | \leq t \bigr) = \sum_{k \in\mathbb{N}}
\mathbb{P}(X=x_{2k+1}) \mathbh{1}_{\delta_k \leq t} \leq\sum_{k
\mbox{ s.t. }2^{ - k
\lambda} \leq t} 2^{-k-1} \leq t^{1/\lambda} .
\end{equation}
By Lemma 9 of \cite{BarJorMcA2006},
(\ref{eq.local.vs.global.margin.tmp1}) implies the global margin
condition over $\mathcal{F}^{0-1}$ with function $\varphi(x) = C_5
x^{2(\lambda+
1)} $, where $C_5$ only depends on $\lambda$. This implies the first
part of (ii) since $\lambda= \kappa- 1>0$.
For the second part, (\ref{eq.pr.le.local.vs.global.margin.2}) implies
that
\[
\operatorname{var}_P( f_{2k+1} - f^{\star}) \geq p_k (1-q_k) ( 1 -
p_k )
\geq\frac{ p_k (1-q_k) }{2} \geq\frac{p_k}{4} = 2^{-k-3} ,
\]
hence the second part of (ii) holds with $C_6 = C_5 2^{2-3\kappa}$.
\item[(iii)] By (\ref{eq.pr.le.local.vs.global.margin.3}), $
\operatorname{var}_P
(f_{2k} - f^{\star}) = p_k q_k (1 - p_k q_k) \leq p_k q_k = P(f_{2k} -
f^{\star})
$.
\item[(iv)] By (iii), for every $k \in\mathbb{N}$, a local margin condition
holds on $\mathcal{F}_{2k}$ with function $\varphi_{2k}\dvtx x \mapsto
x^2 $. So,
the right-hand side of a strong margin-adaptive oracle inequality is at
most (keeping only even values of $m$) proportional to
\[
\inf_{0 \leq k \leq M_n/2} \biggl\{ P(f_{2k} - f^{\star}) + \frac{\ln(n)}
{n} \biggr\}
\leq2^{-\ln_2(n)-1} + \frac{\ln(n)}{n} \leq\frac{2 \ln(n)}{n} .
\]
Note that the $\ln(n)$ factor may be replaced by a smaller quantity
depending on the framework.
The last statement on global margin adaptivity holds according to (ii),
since $\varphi^{\star}(x) = L(\kappa) x^{2\kappa/(2\kappa-1)}$, where
$L(\kappa)>0$ only depends on $\kappa$.
\end{enumerate}
\section*{Acknowledgements} The authors gratefully acknowledge the
support of the NSF under awards DMS-0434383 and DMS-0707060. The first
author's research was mostly carried out at Univ Paris-Sud (Laboratoire
de Mathematiques, CNRS -- UMR 8628), with the additional support of
Inria Saclay (Select Project). The authors would also like to thank an
anonymous referee for numerous comments that improved the presentation
and some of the results of the paper.

\printhistory

\end{document}